\documentclass[11pt,leqno, twoside]{amsart}
\usepackage[foot]{amsaddr}
\usepackage{amssymb,amsfonts}
\usepackage{amsmath,amsthm,amsxtra}
\usepackage[]{geometry}
\usepackage{epsfig}
\usepackage{slashed}

\usepackage{color}
\usepackage{verbatim}

\numberwithin{equation}{section}

\renewcommand{\Im}{\mathop{\rm Im  \, }}
\newcommand{\Ker}{\mathop{\rm Ker \, }}

\newcommand{\ord}{{\rm ord}}

\newcommand{\Res}{\mathop{\rm Res  \, }}

\newcommand{\cF}{\mathcal{F}}
\newcommand{\cC}{\mathcal{C}}

\newcommand{\Pc}{\mathcal{P}}

\newcommand{\V}{\mathcal{V}}

\newcommand{\CC}{\mathbb{C}}
\newcommand{\FF}{\mathbb{F}}

\newcommand{\ZZ}{\mathbb{Z}}

\newcommand{\fg}{\mathfrak{g}}

\renewcommand{\tilde}{\widetilde}

\newtheorem{theorem}{Theorem}[section]
\newtheorem{definition}[theorem]{Definition}
\newtheorem{lemma}[theorem]{Lemma}
\newtheorem{corollary}[theorem]{Corollary}
\newtheorem{proposition}[theorem]{Proposition}
\newtheorem{conjecture}[theorem]{Conjecture}
\newtheorem*{lemma*}{Lemma}

\theoremstyle{remark}
\newtheorem{remark}[theorem]{Remark}

\newtheorem{examples}[theorem]{Examples}

\renewcommand{\epsilon}{\varepsilon}

\newcommand{\half}{\frac{1}{2}}

\newcommand{\la}{\lambda}

\newcommand{\al}{\alpha}

\newcommand{\ptl}{\partial}

\begin{document}
	
\title{Poisson $\lambda$-brackets for differential-difference equations}


\author
          {Alberto De Sole$^1$,    Victor G.~Kac$^2$, Daniele Valeri$^3$ and Minoru Wakimoto$^4$} 
\address{$^1$ Dipartmento di Matematica, Sapienza Universita di Roma, P-le Aldo Moro 2, 00185 Roma, Italy. Email: desole@mat.uniroma1.it} 
\address{$ ^2$ Department of Mathematics, M.I.T, 
	Cambridge, MA 02139, USA. Email:  kac@math.mit.edu }
\address{$^3$ Yau Mathematical Center, Tsinghua University, 100084 Beijing, China. Email: daniele@math.tsinghua.edu.cn} 
\address{$^4$ 12-4 Karato-Rokkoudai, Kita-ku, Kobe 651-1334, Japan. Email: ~~wakimoto@r6.dion.ne.jp~~~~.
Supported in part by Department of Mathematics, M.I.T.}

\maketitle

	\thispagestyle{empty}

\section*{Abstract} 
We introduce the notion of a multiplicative Poisson $ \la $-bracket, which plays the same role in the theory of Hamiltonian differential-difference equations as the usual Poisson $ \la $-bracket plays in the theory of Hamiltonian PDE. We classify multiplicative Poisson $ \la $-brackets in one difference variable up to order 5. Applying the Lenard-Magri scheme to a compatible pair of multiplicative Poisson $ \la $-brackets of order 1 and 2, we establish  integrability of some differential-difference equations, generalizing the Volterra chain.

\section{Introduction}
The notion of a Lie conformal algebra appeared naturally in the study of commutators of local formal distributions. Namely, expanding the commutator in terms of derivatives of the formal delta-function $ \delta(z-w) = \sum_{n \in \ZZ} w^n z^{-n-1}, $
\begin{equation}\label{e1.1}
[a(z), b(w)] = \sum_{j = 0}^{N} c^j (w) \partial^j_w \delta (z-w) / j!, 
\end{equation}
we may define the $ \la $-bracket as the Fourier transform of \eqref{e1.1}:
\begin{equation}\label{e1.2}
[a(w)_\la b(w)] = \sum_{j = 0}^{N} \frac{\la^j}{j!} c^j (w).
\end{equation}
Then, letting $ \partial = \partial_w,  $ it is easy to see that the $ \la $-bracket \eqref{e1.2} satisfies the following properties \cite{K98}:
\begin{enumerate}
\item[C1] (sesquilinearity) $ [\partial a_\la b] = -\la [a_\la b], \ \partial [a_\la b] = [\partial a_\la b] + [a_\la \partial b],   $
\item[C2] (skewsymmetry) $ [b_\la a ] = -\vphantom{j}_\leftarrow [a_{-\partial-\la} b],  $
\item[C3] (Jacobi identity) $ [a_\la [b_\mu c]] - [b_\mu [a_\la c]] = [[a_\la b]_{\la + \mu} c]. $
\end{enumerate}
Here and thereafter the arrow to the left (resp. right) means that $ \partial  $ should be moved to the left (resp. right).

Recall that a \emph{Lie conformal algebra} is a vector space $ R $ with an endomorphism $ \partial, $ endowed with a $ \la $-bracket
$  R \otimes R \rightarrow R[\la],\, a \otimes b \mapsto [a_\la b]$, such that the axioms C1, C2, and C3 hold. 

Subsequently, in the paper \cite{GK98} the notion of a $ \Gamma $-locality was studied, where $ \Gamma $  is a subgroup of the group of fractional linear transformation $ w \mapsto \frac{aw + b}{c w + d} $. Here we consider the simplest case when $ \Gamma $ is a cyclic group, generated by a transformation $ \gamma. $ Then, instead of
\eqref{e1.1}, we consider the bracket of the form
\begin{equation}\label{e1.3}
[a(z), b(w)]  = \sum_{j = -N}^{N} c^j (w) \delta (z-\gamma^j \cdot w).
\end{equation}
It is easy to see that, letting
\[ a(w)_{(j)} b(w) = c^j (w), \ Sa(w) = \gamma' (w) a(\gamma \cdot w), \]
the following properties hold \cite{GK98}:
\begin{equation}\label{e1.4}
(Sa)_{(j)} b = a_{(j+1)}b, \ S(a_{(j)} b) = (Sa)_{(j)}Sb,
\end{equation}
\begin{equation}\label{e1.5}
b_{(j)}a = -S^j (a_{(-j)}b),
\end{equation}
\begin{equation}\label{e1.6}
a_{(i)}(b_{(j)} c) - b_{(j)} (a_{(i)}c) = (a_{(i-j)}b)_{(j)}c.
\end{equation}
Introducing, in analogy with \eqref{e1.2}, the $ \la $-bracket $ [a_\la b] = \sum_j \la^j (a_{(j)}b), $ the properties \eqref{e1.4}--\eqref{e1.6} can be rewritten as follows:
\begin{enumerate}
\item[M1] (sesquilinearity) $ [S a_\la b] = \la^{-1} [a_\la b], \ S [a_\la b] = [S a_\la Sb],   $
\item[M2] (skewsymmetry) $ [b_\la a ] = -\vphantom{j}_\leftarrow [a_{(S\la)^{-1}} b],  $
\item[M3] (Jacobi identity) $ [a_\la [b_\mu c]] - [b_\mu [a_\la c]] = [[a_\la b]_{\la  \mu} c]. $
\end{enumerate}
 We thus arrive at the following definition.
\begin{definition}\label{def0.1}
A \emph{multiplicative} Lie conformal algebra is a vector space $ R $ with an automorphism $ S, $ endowed with a $ \la $-bracket $ R \otimes R \rightarrow R[\la, \la^{-1}], \ a \otimes b \mapsto [a_\la b], $ such that the axioms M1, M2, and M3 hold.   
\end{definition}

Note that the axioms M1--M3 are obtained from C1--C3 by replacing $ \la + \mu $ by $ \la \mu $ and the derivation $ \partial  $ by the automorphism $ S $, hence the name ``multiplicative''. The multiplicative Lie conformal algebras are classified by pairs $( \fg, S)  $, where $ \fg $ is a Lie algebra and $ S $ its ``admissible'' automorphism due to the following
\begin{remark}\label{rem1.1}
\cite{GK98} A multiplicative Lie conformal algebra $ R $ carries a Lie algebra structure with the bracket $ [a,b] = a_{(0)}b $ (the coefficient of $ \la^0 $ in the $ \la $-bracket).
Furthermore, $ S $ is an automorphism of this Lie algebra, satisfying the following admissibility property:  
\begin{equation}\label{e1.7}
 [S^n a, b] = 0 \text{ for all but finitely many } n \in \ZZ.
\end{equation}
Conversely, given a Lie algebra $ \fg $ with an automorphism $ S $, satisfying \eqref{e1.7}, we can introduce on $ \fg $ the associated structure of a multiplicative Lie conformal algebra, letting 
\begin{equation}\label{e1.8}
[a_\la b] = \sum_{n \in \ZZ} \la^n [S^n a, b].
\end{equation}
\end{remark}

The purpose of the present paper is to study the multiplicative Poisson vertex algebras (PVA) and explain their role in the theory of Hamiltonian differential difference equations. Our main idea is that the multiplicative PVA play the same role in the theory of Hamiltonian differential difference equations as the usual PVA play in the theory of Hamiltonian PDE (see \cite{BDSK09} for the latter). Thus, this paper may be viewed as a development of ideas of Boris Kuperschmidt \cite{Ku85}, to whom this paper is dedicated. 

\begin{definition}\label{def1.2}
A multiplicative PVA is a unital commutative associative algebra $ \mathcal{V} $ with an automorphism $S$, endowed with a multiplicative Lie conformal algebra $\lambda$-bracket $ \{ a_\la b  \} $, such that one has
\begin{enumerate}
	\item[L1] (left Leibniz rule) $ \{ a_\la bc \} = \{ a_\la b \} c + b \{ a_\la c \}.   $
\end{enumerate}
Using skewsymmetry M2, we deduce
\begin{enumerate}
	\item[L2] (right Leibniz rule) $ \{ a b_\la c \} = \{ a_{\la S} c \}_{\rightarrow} b +  \{ b_{\la S } c \}_{\rightarrow} a.   $
\end{enumerate}
\end{definition}
\begin{remark}\label{rem1.2}
	Remark \ref{rem1.1} extends to any multiplicative PVA $ \mathcal{V} $. Namely, the Lie algebra bracket on $ \mathcal{V}, $ defined by Remark \ref{rem1.1}, together with the associative commutative multiplication on $ \mathcal{V} $, is a Poisson algebra with an automorphism $ S $, for which the Poisson bracket satisfies \eqref{e1.7}.
	\end{remark}

The first main result of the paper is the classification of multiplicative PVA
$\lambda$-brackets of order $N\leq 5$ on the space of functions $ \mathcal{V}_1 $ in one difference variable $ u. $ Note that, due to skewsymmetry M2, such a $ \la $-bracket of order $ N $ has the form 
\begin{equation}\label{e1.9}
  \{ u_\la u  \} = \sum_{j=1}^{N} (\la^j - (S \la)^{-j} ) f_j, \ f_j \in \V,\, f_N\neq 0,
\end{equation}
and, extending \eqref{e1.9} to $ \V $ using the sesquilinearity M1 and the Leibniz  rules L1 and L2 (or rather the master formula \eqref{e2.02}), the obtained $ \la $-bracket satisfies skewsymmetry.

Let $ u_n = S^n (u),  $ so that $ u_0 = u. $ The first series of examples of multiplicative Poisson $ \la $-brackets on $ \V_1 $ (i.e., satisfying, in addition, the Jacobi identity M3) is given by 
\begin{equation}\label{e1.10}
\{ u_\la u \}_{k,g} = \la^k g(u) g(u_k) - \la^{-k} g(u) g(u_{-k}),
\end{equation}
where $ g(u) \in \V_1.$ All these $ \la $-brackets are compatible, i.e. their arbitrary linear combination $ \sum_{j=1}^{N} c_j \{ u_\la u  \}_{j,g}, $ where $ c_j $'s are constants, is again a multiplicative Poisson $ \la $-bracket, called the multiplicative Poisson $ \la $-bracket of \emph{general type}.

The second series of examples, called the \emph{complementary type} multiplicative $ \la $-brackets, is of the form \eqref{e1.9} with $ N \geq 2, $ where 
\begin{equation}\label{e1.11b}
\begin{aligned}
 f_1 & = f_2 = \ldots = f_{N-3} = 0, \\
 f_{N-2} & = g(u) g(u_{N-2}) F_1 (u) \ldots F_{N-1} (u_{N-2}), \\
 f_{N-1} & = g(u) g(u_{N-1}) (\epsilon^{N-2} F_1 (u) \ldots F_{N-1} (u_{N-2}) + \epsilon^{2-N} F_1 (u_1) \ldots F_{N-1} (u_{N-1}) ), \\
 f_N & = g(u) g(u_N) F_1 (u_1) \ldots F_{N-1} (u_{N-1}),\\
\end{aligned}
\end{equation}
$ g(u), F_j (u)  $ are non-zero elements of $ \V_1 $ and $ \epsilon $ is a constant, such that
\begin{equation}\label{e1.12b}
g(u) F'_j (u) = \epsilon^{j-1} F_j (u), \ j = 1, \ldots, N-1, \ \epsilon^{N-1} = -1.
\end{equation}
We denote this Poisson $ \la $-bracket by $ \{  ._\la .  \}_{N, g, \epsilon} \, . $ 
For example,
\begin{equation}\label{e1.13a}
\{ u_\la u \}_{2,g,-1} = (\la -(\la S)^{-1}) (g(u)g(u_1) (F(u)+F(u_1)) + (\la^2 - (\la S)^{-2}) g(u) g(u_2) F(u_1), 
\end{equation}
where $ F'(u) g(u) = F(u), F'(u) \neq 0. $

Next, note that, given $ n \in \ZZ_{\geq 1}, $ replacing in \eqref{e1.9} $ \la^j $ by $ \la^{nj}, \ S^j  $ by $ S^{nj} $ and $ f_j = f_j (u, u_1, u_2, \ldots) $ by $ f_j (u, u_n, u_{2n}, \ldots), $ we obtain from a multiplicative Poisson $ \la $-bracket $ \{ ._\la . \} $ the $ n $-\emph{stretched} multiplicative Poisson $ \la $-bracket $ \{ ._\la . \}^{(n)}. $ Its order is $ nN. $

It is straightforward to check that all the above examples indeed satisfy the Jacobi identity M3 for $a=b=c=u$ (by Proposition \ref{prop2.1} it follows that these examples are multiplicative Poisson vertex algebras). We prove that any multiplicative Poisson $ \la $-bracket on $ \V_1 $ of order $ \leq 5 $ is one of the following:
\begin{enumerate}
\item[(i)] general type, 
\item[(ii)] constant multiple of the complementary type,
\item[(iii)] linear combination of the $ \la $-brackets $ \{ . _{\la} . \}_{N, g, \epsilon} $ and $ \{ . _{\la} . \}_{1,g} $ where $ N = 2 $ or 3,
\item[(iv)] linear combination of the complementary type $ \la $-bracket
  $ \{ ._\la .  \}_{2,g,-1} $ and the following $ \la $-bracket of order 4
\[ 
\begin{aligned}
f_1 & = 0, \ f_2 = g(u) g(u_2) F(u) F(u_1)^{-1} F(u_2), \\
f_3 & = g(u) g(u_3) (F(u) F^{-1} (u_1) F(u_2) + F(u_1) F^{-1}(u_2)F(u_3)),\\
f_4 & = g(u) g(u_4) F(u_1) F^{-1} (u_2) F(u_3),
\end{aligned} \]
where $ g(u)F'(u) = F(u) $ and $ F(u) \neq 0 $.
\item[(v)] linear combination of the 2-stretched $ \la $-brackets
  $ \{ ._\la .  \}^{(2)}_{2,g, -1} $ and $ \{ ._\la . \}^{(2)}_{1,g}, $
\item[(vi)] constant multiple of the following $ c $-deformed complementary type multiplicative $ \la $-bracket of order 4, where $ g(u), \ F_j(u),  $ and $ \epsilon $ are as in \eqref{e1.12b} for $ N=4 $, $\epsilon\neq -1$:
\[ 
\begin{aligned}
f_1 & = g(u) g(u_1) (c F_1(u) F_3 (u_1) - c^2) , \\
f_2 & =  g(u)g(u_2) (F_1(u) F_2 (u_1) F_3(u_2) + c ( \epsilon^2 F_1 (u) F_3 (u_1) + \epsilon^{-2} F_1 (u_1) F_3(u_2)  )  ), \\
f_3 & = g(u)g(u_3) ( \epsilon^2 F_1 (u) F_2(u_1) F_3(u_2) + \epsilon^{-2} F_1(u_1) F_2(u_2) F_3(u_3) + cF_1(u_1) F_3 (u_2) ),  \\
f_4 & = g(u) g(u_4)F_1 (u_1) F_2 (u_2) F_3 (u_3) . \\
\end{aligned} \]
\item[(vii)] linear combination of the order 2 general type  $\lambda$-bracket
  $\{u_\la u\}_{1,g} +\{u_\la u\}_{2,g}$ and the following $\la$-bracket of order 5:
\[  \begin{aligned}
f_1 & = g(u) g(u_1) F(u) G(u_1), \\
f_2 & = - g(u)g(u_2) (\epsilon F(u) G (u_1)  + \epsilon^{-1} F (u_1) G (u_2)), \\
f_3 & = g(u)g(u_3) (\epsilon^{-1} F(u) G (u_1) +F(u_1)G(u_2) + \epsilon F (u_2) G (u_3)), \\
f_4 & = - g(u)g(u_4) (\epsilon F(u_1) G (u_2)  + \epsilon^{-1} F (u_2) G (u_3)), \\
f_5 & = g(u)g(u_5) F(u_2) G (u_3), \\
  \end{aligned}\]
  where $\epsilon$ is a primitive 3rd root of 1, $g(u)\neq 0$, and
  $g(u)F'(u)=\epsilon F(u)$ and $g(u)G'(u)= G(u)$.
\item[(viii)]  constant multiple of the following multiplicative $ \la $-bracket of order 5, attached to non-zero functions $ F(u), \ g(u) \in \V, $ such that $ g(u)F'(u) =  F(u) $ and a constant  c , given by
\[ 
\begin{aligned}
f_1 & =  g(u) g(u_1) (F(u) F(u_1) + c(F(u) + F(u_1)) + c^2), \\
f_2 & =  -g(u) g(u_2) (F(u)F(u_1) + F(u_1)F(u_2) + c(F(u) + F(u_1) +F(u_2)) + c^2), \\
f_3 & = g(u) g(u_3) (F(u) F(u_1) + F(u_1) F(u_2) + F(u_2) F(u_3) + c(F(u_1) + F(u_2))),  \\
f_4 & =  -g(u)g(u_4)(F(u_1)F(u_2) + F(u_2) F(u_3) + cF(u_2)), \\
f_5 & = g(u) g(u_5) F(u_2) F(u_3).   \\
\end{aligned} \]
\end{enumerate}

We give a detailed proof of this classification for $ N = 1, 2, 3 $ and 4 (Theorems
\ref{th2.1}, \ref{th8.1}, and \ref{th10.1}). The proof for $ N =5 $ (under the same assumptions on $ \V $ as in Theorems \ref{th2.1}, \ref{th8.1}, and \ref{th10.1}) is similar, but involves much more computations, which are skipped.

Thus, we see that, in spite of many analogies, the classification of multiplicative Poisson $\la$-brackets
is radically different from that of ordinary Poisson $\la$-brackets, see \cite{DSKW10}. 

%
%

Multiplicative PVA $ \mathcal{V} $ gives rise to a Hamiltonian differential-difference equation as follows. Denote by 
\[ \smallint : \V \rightarrow \bar{\V} := \V / (S-1) \V \]
\noindent the canonical quotient map. Then it is immediate to see that the following key lemma holds. 

\begin{lemma}\label{lem1.1}
Formula
\begin{equation}\label{e1.11}
\{ \smallint f, \smallint g   \} = \smallint \{f_\la g \} |_{\la = 1}
\end{equation}
\noindent endows $ \bar{\V} $ with a well-defined Lie algebra structure, and the formula
\begin{equation}\label{e1.12}
\{ \smallint f, g   \} =  \{f_\la g \} |_{\la = 1}
\end{equation}
defines a representation of the Lie algebra $ \bar{\V} $ by derivations of the multiplicative PVA $ \V, $ which commute with $ S $.
\end{lemma}

Choosing a Hamiltonian functional $ \int h \in \bar{\V}, $ we define the corresponding Hamiltonian equation
\begin{equation}\label{e1.13}
\frac{du}{dt} = \{ \smallint h, u \}, \ u \in \V. 
\end{equation}
A Hamiltonian function $ \int h_1 $ is called an integral of motion of this equation if $ \int \frac{dh_1}{dt} = 0 $ in virtue of \eqref{e1.13}, i.e. $ \{ \int h, \int h_1 \} = 0. $ The equation \eqref{e1.13} is called \emph{integrable} if it has infinitely many integrals of motion in involution, i.e. if $ \int h $ is contained in an infinite-dimensional abelian subalgebra of the Lie algebra $ \bar{\V} $ with bracket \eqref{e1.11}.

The most famous example of a differential-difference equation is the Volterra chain:
\begin{equation}\label{e1.14}
\frac{du}{dt} = u(u_1 - u_{-1})
\end{equation}
\noindent (applying $ S^n $ to both sides, we obtain its more traditional form $ \frac{du_n}{dt} = u_n (u_{n+1} - u_{n-1}), \ n \in \ZZ$). This equation can  be written in a Hamiltonian form \eqref{e1.13} with respect to the order 1 general type Poisson $ \la $-bracket $ \{ u_\la u \}_{u, 1} = \la uu_1 -\la^{-1} uu_{-1} $ by choosing $ \int u $ as a Hamiltonian functional. Moreover, equation \eqref{e1.14} is bi-Hamiltonian, i.e. it can be written in the form \eqref{e1.13} with a different, complementary type order 2 multiplicative Poisson $ \la $-bracket $ \{ u_\la u \}_{2, u, -1} $ and the Hamiltonian functional $ \half \int \log u. $
These two multiplicative Poisson $ \la $-brackets are compatible, hence the Lenard-Magri type scheme could be applied. We prove that indeed in this case the Lenard-Magri sequence can be infinitely extended, proving thereby that the Volterra chain \eqref{e1.14} has infinitely many linearly independent integrals of motion, i.e. is integrable (this is a special case of Proposition \ref{prop7.3} for $ F(u) = u $).



In the paper we develop the related theory of the variational complex. Of course some of it  has been done already in \cite{Ku85}.

Our work was stimulated by the book \cite{Ku85} and the paper \cite{KMW13}, from which we learned the basics of the subject. We are grateful to S. Carpentier for pointing out that the equations obtained from compatible Poisson $\la$-brackets of higher order via the Lenard-Magri scheme
reduce to the Volterra equation.

\section{Classification of multiplicative PVA of order $\leq 2$ in one variable}
Analogously to the case of ordinary PVA \cite{BDSK09}, we shall work in the framework of an algebra of ``difference'' functions $ \V. $ The simplest and most important example is the algebra of difference polynomials $ \Pc_\ell $ in $ \ell $ variables $ u^1, \ldots, u^\ell. $ This is the algebra of polynomials over $ \FF $
(our base field) in the variables $ u^i_n, $ where $ i \in I = \{ 1, \ldots, \ell \} $ and $ n \in \ZZ. $ This algebra carries an automorphism $ S, $ defined by $ Su^i_n = u^i_{n+1}, \ i \in I, \ n \in \ZZ.  $ It satisfies the following relation
\begin{equation}\label{e2.01}
  S \frac{\partial}{\partial u^i_n} = \frac{\partial}{\partial u^i_{n+1}} S,
  \,\,i\in I,  \,n \in \ZZ.
\end{equation}
\begin{definition}\label{def2.1}
  An algebra of difference functions in $ \ell $ variables $u^1,..., u^\ell$ is a commutative associative unital algebra $ \V, $ containing
  $ \Pc_\ell $,
  and endowed with commuting derivations $\frac{\partial}{\partial u^i_n}$ extending those on $ \Pc_\ell $, and an automorphism $S$ extending that on $ \Pc_\ell $, such that the following two properties hold:
\begin{enumerate}
\item[(i)] $ \frac{\ptl f}{\ptl u^i_n} = 0 $  for all but finitely many pairs $ (i,n); $
\item[(ii)] formula \eqref{e2.01} holds. 
	\end{enumerate}
\end{definition}
Let
\[ \cF = \left\{ f \in \V \, \middle| \, \frac{\ptl f}{\ptl u^i_n} = 0 \text{ for all } (i,n)   \right\} \, .
\]
Let $ \cC = \{ f \in \V \, | \,  Sf = f \} $ be the subalgebra of constants. Note that, by \eqref{e2.01} and axiom (i), $\cC \subset \cF $ and $ S \cF = \cF. $

The proof of the following Proposition is the same as for the ordinary PVA, see \cite{BDSK09}, Theorem 1.15.
\begin{proposition}\label{prop2.1}
Let $ \V $ be an algebra of difference functions in $ \ell $ variables $ u^1, \ldots, u^\ell. $ For each pair $ i,j \in I $ choose $ \{ u^i_\la u^j \} \in \V [\la, \la^{-1}]. $ Then
\begin{enumerate}
\item[(a)] The \emph{master formula}
\begin{equation}\label{e2.02}
\{ f_\la g \} = \sum_{\substack{i,j \in I \\ m,n \in \ZZ}} \frac{\ptl g}{\ptl u^j_n} (\la S)^n \{ u^i_{\la S} u^j \}_{\rightarrow} (\la S)^{-m} \frac{\ptl f}{\ptl u^i_m} 
\end{equation}
defines a multiplicative $ \la $-bracket on $ \V $, satisfying the sesquilinearity
M1 and the Leibniz rules L1 and L2 (see introduction), which extends the given $ \la $-brackets on the generators $ u^i, \ i \in I. $
\item[(b)] Formula \eqref{e2.02} satisfies the skewsymmetry M2, iff it holds on each pair of generators. 
\item[(c)] Assuming that the skewsymmetry holds on each pair of generators, the multiplicative $ \la $-bracket \eqref{e2.02} satisfies the Jacobi identity M3, iff it holds on any triple of generators.  
\end{enumerate}
\end{proposition}
\begin{remark}\label{rem2.1}
The master formula defines a unique multiplicative $ \la $-bracket on $ \Pc_\ell, $ satisfying the sesquilinearity and the Leibniz rules, with the given $ \{u^i_\la u^j  \}, \ i,j \in I.  $
This uniqueness holds also if $\V$ is obtained from $\Pc_\ell$ by adjoining solutions
of polynomial or differential equations with coefficients in $\V$, like inverses of non-zero elements or exponentials.
\end{remark}
\begin{definition}\label{def2.2}
A $\lambda$-bracket, defined by the master formula \eqref{e2.02}, is called a multiplicative Poisson $\lambda$-bracket if it defines a structure of a multiplicative PVA on $\V$.
\end{definition}
\begin{theorem}\label{th2.1}
  Let $ \V $ be an algebra of difference functions in one variable $ u $ without zero divisors and such that $\cC=\cF$ is a field. Then any multiplicative Poisson $ \la $-bracket on $ \V $ of order $ \leq 2 $
  is either of general type (= linear combination of $\lambda$-brackets \eqref{e1.10} with $ k \leq 2 $), or is a linear combination of the $\lambda$-bracket
  $\{._{\lambda}.\}_{2,g,-1}$
  of complementary type, given by
  \eqref{e1.13a},
  and the $\lambda$-bracket
  $\{._{\lambda}.\}_{1,g}$ of general type, given by \eqref{e1.10}.
\end{theorem}
\begin{proof}
Let $ \{ u_\la u  \} = \sum_{k \in \ZZ} \la^k f_k, \  f_k \in \V. $ Using the master formula, the Jacobi identity M3 (from the Introduction) for $ a = b = c = u $ becomes:
\begin{equation}\label{e2.03}
\begin{aligned}
& \sum_{i,k} \la^{i+k} (S^i f_k) \frac{\ptl}{\ptl u_i} \sum_j \mu^j f_j - \sum_{i,k} \mu^{i+k} (S^i f_k) \frac{\ptl }{\ptl u_i} \sum_j \la^j f_j \\
  = & \sum_{i,k} (\la \mu)^{i+k} f_k S^{i+k} ( \frac{\ptl}{\ptl u_{-i}}
  \sum_j \la^j f_j ).
\end{aligned}
\end{equation}
Note that for the $\la$-bracket, satisfying skewsymmetry, the coefficients of
$\la^m \mu^n$, $\la^n \mu^m$, $\la^{-n}\mu^{m-n}$ and $\la^{m-n}  \mu^{-n}$ in
\eqref{e2.03} give the same equation on the $f_j$'s. Hence all the equations come from the coefficients of $\la^m \mu^n$ with $0<m<n$. Hence for such a pair $(m,n)$ for the
multiplicative Poisson $\la$-bracket of degree $N$ the corresponding term
appears in \eqref{e2.03} iff $0<m\leq N$ and $m<n\leq m+N$. The number of such pairs is $N^2$.

We now use the following lemma.
\begin{lemma}\label{lem2.3}
	Identity \eqref{e2.03} and skewsymmetry imply that $ f_k = f_k (u, u_1, \ldots, u_k) $ for $ k>0. $
\end{lemma}
\begin{proof}
  Let
  $ N = \max \{ i \, | \, f_i \neq 0 \}, i_k = \max \{ i \, | \,
  \frac{\ptl f_k}{\ptl u_i}\neq 0 \}, $
  and suppose that $ i_k \geq k+1. $ Computing the coefficient of $ \la^{i_k +N} \mu^k $ in \eqref{e2.03}, we obtain: $ (S^{i_k} f_N) \frac{\ptl f_k}{\ptl u_{i_k}} - 0 = 0, $ hence $ \frac{\ptl f_k}{\ptl u_{i_k}} = 0, $ a contradiction. Hence $ \frac{\ptl f_k}{\ptl u_i} = 0 $ for $ i > k $ if $ k>0. $ In a similar way we prove that $ \frac{\ptl f_{-k}}{\ptl u_i} = 0 $ for $ i < -k $ if $ k>0. $ The lemma follows due to the skewsymmetry relation $ f_k = -S^k f_{-k}, \ k \in \ZZ.  $
\end{proof}

By \eqref{e1.9} and Lemma \ref{lem2.3}, the $ \la $-bracket $ \{ u_\la u \} $ of order $N $ has the form
\begin{equation}\label{e2.00}
  \{ u_\la u \} = \sum_{k=1}^N(\la^k - (\la S)^{-k}) f_k(u,u_1,...,u_k).
\end{equation}

Let now $N\leq 2$. Then \eqref{e2.03} is equivalent to the following four equations on $ f_1 = f_1 (u, u_1)$ and $ f_2 = f_2 (u, u_1, u_2), $ which correspond to coefficients of $\la^2\mu^4,
\la^2\mu^3,\la \mu^3, \la \mu^2$, respectively:
\begin{equation}\label{e2.04}
f_2 S^2  \frac{\ptl f_2}{\ptl u}= (S^2 f_2) \frac{\ptl f_2}{\ptl u_2},
\end{equation}
\begin{equation}\label{e2.05}
  (S^2 f_1) \frac{\ptl f_2}{\ptl u_2} + (Sf_2) \frac{\ptl f_2}{\ptl u_1} = f_2 S^2
  \frac{\ptl f_1}{\ptl u},
\end{equation}
\begin{equation}\label{e2.06}
(Sf_2) \frac{\ptl f_1}{\ptl u_1} = f_1 S  \frac{\ptl f_2}{\ptl u}  + f_2 S  \frac{\ptl f_2}{\ptl u_1},
\end{equation}
\begin{equation}\label{e2.07}
(S f_1) \frac{\ptl f_1}{\ptl u_1} + (Sf_1) \frac{\ptl f_2}{\ptl u_2} = f_1 \frac{\ptl f_2}{\ptl u} + f_1 S \left( \frac{\ptl f_1}{\ptl u}  \right) - f_2 \frac{\ptl f_1}{\ptl u} + f_2 S  \frac{\ptl f_1}{\ptl u_1} .
\end{equation}

First, consider the case $N=2$, i.e. $ f_2 \neq 0. $ Note that, by \eqref{e2.04},
\begin{equation}\label{e2.08}
\frac{\ptl f_2}{\ptl u_2} / f_2 = S^2  (\frac{\ptl f_2}{\ptl u} / f_2) .
\end{equation}
Since the LHS (resp. RHS) of this equation is a function of $ u, u_1, u_2  $ (resp. $ u_2, u_3, u_4 $), we conclude that both sides are functions of $ u_2. $ It follows from \eqref{e2.08} that $ \log f_2 $ is a sum of a function in $ u_2 $ and a function in $ u, u_1 $ (resp. a sum of a function in $ u $ and a function in $ u_1, u_2 $). Hence 
\[ f_2 = p(u_2) \varphi (u, u_1) = g(u) \psi (u_1, u_2).  \]
It follows that $ f_2 / p(u_2) $ is independent of $ u_2, $ hence $ \psi (u_1, u_2) / p(u_2) = h(u_1). $ Thus 
\[ f_2 = g(u) h(u_1) p(u_2). \]
It follows that
\[ \frac{\ptl f_2}{\ptl u_2} / f_2 = \frac{\ptl }{\ptl u_2} \log p(u_2), \ \frac{\ptl f_2}{\ptl u} / f_2 = \frac{\ptl}{\ptl u} \log g(u). \]
Substituting this in \eqref{e2.08}, we obtain, using \eqref{e2.01}, 
\[ \frac{\ptl}{\ptl u_2} \left( \log p(u_2) - \log g(u_2) \right) = 0. \]
Hence $ p (u_2) = cg(u_2), $ where $ c $ is a non-zero constant. Absorbing this constant in $ h(u_1), $ we obtain:
\begin{equation}\label{e2.09}
f_2 (u,u_1, u_2) = g(u) h(u_1) g(u_2).
\end{equation}
Note that, conversely, \eqref{e2.09} implies \eqref{e2.04}.

Next, we analyze equation \eqref{e2.05}. Substituting in it \eqref{e2.09} and dividing both sides by $ g(u) h(u_1), $ we obtain
\begin{equation}\label{e2.10}
(S^2 f_1) \frac{\ptl g(u_2)}{\ptl u_2} + g(u_2) h(u_2) g(u_3) \frac{g(u_1)}{h(u_1)} \frac{\ptl h(u_1)}{\ptl u_1} = g(u_2) S^2 \frac{\ptl f_1}{\ptl u}. 
\end{equation}
Since the first term in the LHS and the RHS are independent of $ u_1, $ we conclude that the second term in the LHS is independent of $ u_1. $ Hence\begin{equation}\label{e2.11}
\frac{g(u)}{h(u)} \frac{\ptl h(u)}{\ptl u} = a \in \cC. 
\end{equation}
First consider the case when $ h'(u) \neq 0. $ Then we can see from \eqref{e2.09} that 
\begin{equation}\label{e2.12}
 f_2 = g(u) g(u_2) h(u_1), \text{ where } g(u) = a \frac{h(u)}{h'(u)},\,a\neq 0, 
\end{equation}
 which is the coefficient of $ \la^2 $ in \eqref{e1.10} with $ F = h. $
 Substituting \eqref{e2.11} in \eqref{e2.10}, we obtain the following equation, to which $ S^2 $  is applied:
\[ f_1 (u,u_1) g' (u) + a g(u) h(u) g(u_1) = g(u) \frac{\ptl f_1 (u,u_1)}{\ptl u}. \]
Dividing both sides by $ g(u)^2, $ we obtain:
\[ \frac{\ptl}{\ptl u} \frac{f_1 (u,u_1)}{g(u)} = a \frac{h(u)}{g(u)} g(u_1) = h' (u) g(u_1) \]
Integrating by $ u $ and multiplying by $ g(u),  $ we obtain
\begin{equation}\label{e2.13}
f_1 (u, u_1) = g(u) h(u) g(u_1) + g(u) A (u_1).
\end{equation}

Next, we use equation \eqref{e2.06}, in which we substitute \eqref{e2.12} to get:
\[ g(u_1) \frac{\ptl f_1 (u,u_1)}{\ptl u_1} = f_1 (u, u_1)
\frac{\ptl g(u_1)}{\ptl u_1} + a g(u) g(u_1) h(u_1) \, . \]
Dividing both sides by $ g(u_1)^2, $ we obtain 
\[ \frac{\ptl }{\ptl u_1} \left( \frac{f_1 (u,u_1)}{g (u_1)} \right) = a g(u) \frac{h(u_1)}{g(u_1)}, \]
and, using \eqref{e2.11}, we get
\[ \frac{\ptl}{\ptl u_1} \left( \frac{f_1 (u,u_1)}{g(u_1)} \right) = g(u) h'(u_1) \, . \]
Integrating by $ u_1, $  we have, after multiplying both sides by $ g(u_1), $
\begin{equation}\label{e2.14}
f_1 (u, u_1) = g(u) h(u_1) g(u_1) + g(u_1) B(u).
\end{equation}
Equating the RHS's of \eqref{e2.13} and \eqref{e2.14} and dividing by $ g(u) g(u_1),  $ we obtain:
\[ \frac{A (u_1) - g(u_1)h(u_1)}{g(u_1)} = \frac{B(u) - g(u) h(u)}{g(u)} \, . \]
It follows that both sides are equal to $ c \in \cC, $ hence
\[ A(u_1) = g(u_1) h(u_1) + c g(u_1), \]
and, by \eqref{e2.13}, we see that
\[ f_1 (u,u_1) = g(u) g(u_1) (h(u) + h(u_1) + c), \] 
completing the case when $ h'(u) \neq 0. $
 
Finally, consider the case $ h'(u) = 0,  $ i.e. $ h(u)  \in \cC.  $ This case also includes the case $ f_2 = 0 $ by letting $ h = 0 $ in \eqref{e2.09}.
In this case \eqref{e2.10} and \eqref{e2.12} become:
\[ \frac{\ptl}{\ptl u_i} (\log f_1) = \frac{\ptl}{\ptl u_i} \log g(u_i) \text{ for } i = 0, 1, \]
so that $ \log f_1 = \log g(u) + C(u_1) = \log g(u_1) + D(u). $ Hence $ f_1 = c_1 g(u) g(u_1), \ c_1 \in \cC, $ and by \eqref{e2.09}, $ f_2 = c_2 g(u) g(u_2)\,, c_2\in \cC.  $

It is straightforward to check that the $f_1$ and $f_2$, obtained above, do satisfy equations \eqref{e2.04}--\eqref{e2.07}, completing the proof of Theorem
\ref{th2.1}.
\end{proof}

\begin{remark}\label{rem2.7}
If we drop the assumption that $ \cC = \cF $ in Theorem \ref{th2.1}, the classification of multiplicative Poisson $ \la $-brackets of order $ \leq 2 $ is similar, but a little different. Namely, one has the following two possibilities:
\begin{enumerate}
\item[(i)] (general type)
\[ f_j = c_jg S^j (g), \ j=1, 2, \,\hbox{where}\,  c_j \in \cF, \ g \in \V, \, \frac{\ptl g}{\ptl u_i}=0 \, \hbox{for}\,  i \geq 1; \]
\item[(ii)] (complementary type)
\[ f_1 = gS(g) ( \frac{S^{-1}(a)}{a} F + S ( \frac{S(a)}{a} F) +\frac{c}{aS(a)} ), \quad  f_2 = gS^2 (g) S(F),  \]
where $g, F\in \V, $ $a\in \cF,\, c\in \cC,\, \ ag \neq 0,\, \frac{\ptl g}{\ptl u_i} = 0 = \frac{\ptl F}{\ptl u_i} $ for $ i \geq 1, \ g \frac{\ptl F}{\ptl u} = aF$.
\end{enumerate}
\end{remark}

\begin{remark}\label{rem2.8}
Given a Lie algebra $ \fg $, we can construct a multiplicative Lie conformal algebra $ \text{Cur } \fg = \FF [S, S^{-1}]\fg $ with the multiplicative $ \la $-bracket 
\[ [a_\la b] = [a,b]\, \text{ for }\, a,b \in \fg, \]
extended by sesquilinearity M1. We conjecture that a finite rank over $ \CC [S, S^{-1}] $ simple multiplicative Lie confornal algebra is isomorphic to $ \text{Cur } \fg $ for some simple finite-dimensional Lie algebra $ \fg. $ For example, it is easy to check that any rank 1 multiplicative Lie conformal algebra is trivial. Indeed, if $ [u_\la u] = f(\la, S)u $ for some $ f(\la, S) \in \FF [\la, \la^{-1}. S, S^{-1}], $ by the sesquilinearity M1, the Jacobi identity for $ a = b= c = u $ reads:
\begin{equation}\label{e2.15}
f(\mu, \la S) f(\la, S) - f(\la, \mu S) f (\mu, S) = f(\la, (\la \mu)^{-1}) f(\la \mu, S).
\end{equation}
Suppose that $ f \neq 0,  $ and let $ n $ be the maximal power of $ S, $ appearing in $ f $. Then it is immediate to see that the maximal power of $ S $ which occurs in the LHS (resp. RHS) of \eqref{e2.15} is $ 2n  $ if $ n \neq 0 $ and is negative if $ n = 0 $ (resp. is $ n $), a contradiction. 
\end{remark}

In the sequel we shall need the following 
\begin{lemma}
  \label{lem2.7}
  Let $\V$ be an algebra of difference functions, such that the subalgebra of
  constants $\cC$ is a domain and the subalgebra $\cF$ contains no eigenvectors of $S$ other than from $\cC$. Let $P(x)\in \cC[x,x^{-1}]$ be such that $P(-1)\neq 0$. Then the kernel of $P(S)$ in $\V$ is zero.
\begin{proof}
  It suffices to show that the kernel of $ S + a $ is zero for any constant $ a \neq -1. $  Let $ f \in \V  $ be outside of $\cF$ and
   such that $ (S+a)f = 0. $ Then there exists $i$,
  such that $ \frac{\ptl f}{\ptl u^i_n} \neq 0 $ for some integer $n . $
  Take maximal such $n$.
  By \eqref{e2.01} we have:
	\[ S \frac{\ptl}{\ptl u^i_n} = \frac{\ptl}{\ptl u^i_{n+1}} (S+a) - a \frac{\ptl}{\ptl u^i_{n+1}} \, . \]
	Applying both sides to $ f, $ we obtain 
	$ S \frac{\ptl f}{\ptl u^i_n} = 0, $
	a contradiction.
  \end{proof}
\end{lemma}

\section{Evolution difference equations and related notions}

An \emph{evolution difference equation} over an algebra of difference functions $ \V $ in $ \ell $ variables $u=\{ u^i\}_{i\in I} $ is a system of equations of the form
\begin{equation}\label{e3.01}
\frac{du}{dt} = P,\, \text{ where } P = (P^i)_{i \in I} \in \V^\ell.
\end{equation}
Applying $ S^n $ to both sides, we obtain a system of ordinary differential equations
on all $ u^i_n: $
\[ \frac{du^i_n}{dt} = S^n (P^i),\, i \in I, n \in \ZZ.  \]

Recall that elements of $ \bar{\V} = \V / (S-1) \V $ are called \emph{Hamiltonian functionals} and are denoted by $ \int f, f \in \V. $ Such an element is called an \emph{integral of motion} of \eqref{e3.01} if $ \frac{d}{dt} \int f = 0 $ in virtue of \eqref{e3.01}. Using the chain rule, this condition becomes
\begin{equation}\label{e3.02}
\smallint X_P (f) = 0,
\end{equation}
where 
\begin{equation}\label{e3.03}
X_P = \sum_{\substack{n \in \ZZ \\ i \in I}} S^n (P^i) \frac{\ptl}{\ptl u^i_n}
\end{equation}
is the \emph{difference evolutionary vector field,} attached to $ P \in \V^\ell. $ This is a derivation of the algebra $ \V, $ commuting with the automorphism $ S.  $

Since, by definition, $ \int S(f) = \int f, $ we have integration by parts
\begin{equation}\label{e3.04}
\smallint S^n (f)g = \smallint f S^{-n} (g), \ n \in \ZZ.
\end{equation}
Applying integration by parts to \eqref{e3.02}, \eqref{e3.03}, we obtain that $ \int f $ is an integral of motion of the equation \eqref{e3.01} if and only if
\begin{equation}\label{e3.05}
\smallint \frac{\delta f}{\delta u} \cdot P = 0 \,.
\end{equation}
Here $ \frac{\delta f}{\delta u}\in \V^{\ell} $ is the column vector of \emph{difference variational derivatives}
\begin{equation}\label{e3.06}
\frac{\delta f}{\delta u^i} = \sum_{n \in \ZZ} S^{-n} \left( \frac{\ptl f}{\ptl u^i_n} \right),
\end{equation}
and $ P \cdot Q =\sum_{i\in I}P^iQ^i$ stands for the dot product in $\V^\ell$.
In the sequel we shall also need the \emph{difference Frechet derivative} $ D_F(S)$
of $ F \in \V^\ell $, defined as an $ \ell \times \ell $ matrix
difference operator
\begin{equation}
  \label{e3.06a}(D_F(S))_{i,j\in I} = \sum_{n \in \ZZ} \frac{\ptl F^i}{\ptl u^j_n} S^n \,.
\end{equation} 

Recall that the space of difference evolutionary vector fields is closed under the usual Lie bracket $ (P,Q \in \V^\ell) $:
\begin{equation}\label{e3.07}
[X_P, X_Q] = X_{[P,Q]}, \text{ where } [P,Q] = X_P (Q) - X_Q (P).
\end{equation}

The vector field $ X_Q $ is called a \emph{symmetry} of the equation \eqref{e3.01} if $ [X_P, X_Q] = 0. $ This property is equivalent to the \emph{compatibility} of equation \eqref{e3.01} and the equation $ \frac{du}{dt_1} = Q $ in the sense that $ \frac{d}{dt} \frac{d}{dt_1} = \frac{d}{dt_1} \frac{d}{dt} $.


\section{Hamiltonian difference operators}
Let $ \V $ be a unital associative algebra with an automorphism $ S. $ Then $ \V [S, S^{-1}] $ is a unital $ \ZZ $-graded associative algebra with the product $ \circ  $ defined by the relation 
\begin{equation}\label{e4.01}
S \circ f = S(f) S, \ f \in \V, 
\end{equation}
and the $ \ZZ $-grading, defined by $ \deg \V = 0, \ \deg S = 1.  $ The algebra $ \V [S, S^{-1}] $ is called the algebra of \emph{difference operators} over $ \V. $ This algebra carries an anti-involution $ \ast $ defined by 
\[ f^\ast = f \text{ for } f \in \V, \ S^\ast = S^{-1} \, .  \]

Let now $ \V $ be an algebra of difference functions in $ \ell $ variables
$ u^i,\, i\in I=\{1, \ldots, \ell \} $ (see Definition \ref{def2.1}). Let $ \V $ be endowed by a $ \la $-bracket, defined by the master formula \eqref{e2.02}. The $ \ell \times \ell $ matrix difference operator 
\begin{equation}\label{e4.02}
H(S) = ( \{ u^j_S u^i \}_{\rightarrow})_{i,j \in I}
\end{equation}
is called the \emph{Hamiltonian operator} (or Poisson structure), associated with the $ \la $-bracket \eqref{e2.02}, provided that it satisfies the skewsymmetry and Jacobi identities. 

It is immediate to see that \eqref{e1.12} for
$ g = u^j, \ j \in I, $ can be written as 
\begin{equation}\label{e4.03}
\{ \smallint f, u  \} = H(S) \frac{\delta f}{\delta u},
\end{equation}
where $  u $ is the column vector of the $u^i$'s and $ \frac{\delta f}{\delta u} $ is the column vector of difference variational derivatives \eqref{e3.06}.

Applying integration by parts to \eqref{e2.02}, we obtain
\begin{equation}\label{e4.04}
\{ \smallint f, \smallint g \} = \smallint \frac{\delta g}{\delta u} \cdot H(S) \frac{\delta f}{\delta u} \, .
\end{equation}

Of course, formulas \eqref{e4.03}--\eqref{e4.04} are completely analogous to those in the differential case (cf. \cite{Ku85}).

Note also that while the Hamiltonian operator $ H(S) $ is defined via the $ \la $-bracket \eqref{e2.02} by \eqref{e4.02}, conversely, the $ \la $-bracket \eqref{e2.02} can be expressed via $ H(S)  $ by
\begin{equation}\label{e4.05}
(\{ u^j_\la u^i\})_{i,j\in I} = H (\la S) I_{\ell}.
\end{equation}
It follows from \eqref{e4.03} that the Hamiltonian equation \eqref{e1.13} with a $ \la$-bracket, corresponding via \eqref{e4.05} to the Hamiltonian operator $ H(S) $, and a Hamiltonian functional $ \int h \in \bar{\V} $, is an evolution difference equation
\begin{equation}\label{e4.06}
\frac{du}{dt} = H(S) \frac{\delta \smallint h}{\delta u} \, .
\end{equation}

Recall that, by Lemma \ref{lem1.1}, we have a homomorphism of the Lie algebra $ \bar{\V} $ with bracket \eqref{e1.11} to derivatives of $ \V, $ commuting with $ S.  $ By \eqref{e4.03} it is given by the formula
\begin{equation}\label{e4.07}
\smallint f \mapsto X_{H(S) \frac{\delta \smallint f}{\delta u}} \, . 
\end{equation}
Consequently, we obtain the standard
\begin{proposition}\label{prop4.1}
If $ \int f $ is an integral of motion of the Hamiltonian equation \eqref{e4.06}, then the evolutionary vector field $X_{H(S) \frac{\delta}{\delta u} \int f} $ is a symmetry of this equation. 
\end{proposition}

The following proposition translates the properties of the Poisson $ \la $-brackets to that of the corresponding Hamiltonian operators. It is a ``difference'' analogue of Proposition 1.16 from \cite{BDSK09}. 
\begin{proposition}\label{prop4.2}
\begin{enumerate}
\item[(a)] The multiplicative $ \la $-bracket \eqref{e2.02} is skewsymmetric if and only if the associated via \eqref{e4.02} difference operator
  $ H(S)=(H_{ij}(S))_{i,j\in I} $ is skewadjoint:
\[ H(S)^\ast = -H(S), \,\hbox{where}\, (H_{ij}(S) )^\ast = (H_{ji} (S)^\ast) \, . \]
\item[(b)] If the operator $ H(S)  $ is skewadjoint, then the corresponding $ \la $-bracket, defined by \eqref{e2.02} and \eqref{e4.05}, satisfies Jacobi identity if and only if one of the following equivalent conditions holds:
\begin{enumerate}
\item[(i)] the multiplicative $ \la $-bracket on $ V $, associated to $ H(S) $ via \eqref{e4.05}, satisfies the Jacobi identity,
\item[(ii)] the following identity holds for any $ i,j,k \in I: $
\[ \begin{aligned}
&\sum_{t \in I, n \in \ZZ} \left( \frac{\ptl H_{k,j} (\mu)}{\ptl u^{(t)}_n} (\la S)^n H_{t,i} (\la) - \frac{\ptl H_{k,i} (\la)}{\ptl u^{(t)}_n} (\mu S)^n H_{t,j} (\mu) \right) \\
& = \sum_{t \in I, n \in \ZZ} H_{k,t} (\la \mu S)(\la \mu S)^{-n} \frac{\ptl H_{j,i} (\la)}{\ptl u^{(t)}_n},
\end{aligned}
 \]
\item[(iii)] the following identity holds for any $ F, G \in \V^{\ell}: $
\[ \begin{aligned}
&H(S) D_G(S) H(S) F + H(S) D^*_{H(S)F} (S)G - H(S) D_F (S) H(S) G + H(S) D^*_F (S) H (S)G \\
  & = D_{H(S)G} (S) H(S) F - D_{H(S)F} (S) H(S) G,
 \end{aligned}
\]
 where $D_F(S)$ is the difference Frechet derivative of $F\in \V^{\ell}$,  defined by \eqref{e3.06a}.
\end{enumerate}
\end{enumerate}
\end{proposition}
\begin{proof}
(a) is straightforward and equivalence of (i) and (ii) in (b) is clear by \eqref{e4.02}. In order to prove equivalence of (ii) and (iii), note the following identity for any $ F \in \V^\ell: $
\[ (D_{H(S)F} (S)  )_{ij} - (H(S)D_F (S))_{ij} = \sum_{k \in I, n \in \ZZ}  \frac{\ptl H(S)_{ik}}{\ptl u^j_n} F^k S^n \]
Applying both sides to $ H(S)G, $ we obtain for $F,G\in \V^\ell$, $i\in I$:
\begin{equation}\label{e4.08}
  (D_{H(S)F} (S)H(S)G )^i - (H(S)D_F(S)H(S)G)^i = \sum_{j,k \in I, n \in \ZZ}
  (\frac{\ptl H(S)_{ik}}{\ptl u^j_n} F^k) S^n ((H(S)G)^j) \, .
\end{equation}
Denote by \eqref{e4.08}* the identity, obtained from \eqref{e4.08} by applying * to it,and by $ \eqref{e4.08}_F $, obtained from \eqref{e4.08} by substituting $ G $ by $ F. $ Then the identity $ \eqref{e4.08} - \eqref{e4.08}^* + \eqref{e4.08}_F $ shows that identity (iii) is equivalent to the following identity for any $ F, G \in \V^{\ell} $:
\begin{equation}\label{e4.09}
\begin{aligned}
&\sum_{\substack{j,k, t \in I \\ n \in \ZZ} } \frac{\ptl H_{ik} (S)}{\ptl u^j_n} G^k S^n (H_{jt} (S) F^t) - \sum_{\substack{j,k, t \in I \\ n \in \ZZ} } \frac{\ptl H_{it} (S)}{\ptl u^j_n} F^t S^n (H_{jk} (S) G^k)\\
& = \sum_{\substack{j,k, t \in I \\ n \in \ZZ} } H_{ij} (S)S^{-n} (( \frac{\ptl H_{kt} (S)}{\ptl u^{j}_n} F^t) G^k ).
\end{aligned}
\end{equation}
Since identity \eqref{e4.09} holds for every $ F,G \in \V^\ell, $ we can replace in it $ S, $ acting on $ F^t, $ by $ \la $, and $ S, $ acting on $ G^k, $ by $ \mu $, and write it as an identity for polynomials in $ \la $ and $ \mu. $ This shows that identity \eqref{e4.09} is equivalent to (ii).
\end{proof}

\begin{corollary}\label{cor4.3}
Let $ H(S) $ be a Hamiltonian operator, acting on $ \V^\ell. $ Then $ H(S) \V^\ell $ is a subalgebra of $ \V^\ell $ with respect to the bracket \eqref{e3.07}.
\end{corollary}
\begin{proof}
  Since $ D_F(S)G = X_GF, $ the RHS of (iii) in Proposition \ref{prop4.2} is the
  bracket \eqref{e3.07} of $ H(S)F $ and $ H(S)G, $ while the LHS lies in the image of $ H(S).  $
\end{proof}

According to Remarks \ref{rem1.1} and \ref{rem1.2}, there is an alternative language of Poisson brackets on an algebra of difference functions $ \V $ in $ u^i,\, i\in I $, with an automorphism $ S. $
\begin{proposition}\label{prop4.4}
\begin{enumerate}
\item[(a)] Given a Poisson $ \la $-bracket on $ \V $, defined by the $ \la $-brackets 
\[ \{ u^i_\la u^j \} = \sum_{k\in \ZZ} \la^k f^{ij}_k, \quad i,j \in I, \]
let 
\begin{equation}\label{e4.10}
[u^i_m, u^j_n] = S^n f^{ij}_{m-n}, \quad i,j \in I, \ m,n \in \ZZ,
\end{equation}
and extend this to the whole of $ \V $ by the ordinary Leibniz rules:
\begin{equation}\label{e4.11}
[f,g] = \sum_{\substack{i,j \in I \\ m,n \in \ZZ}} \frac{\ptl f}{\ptl u^i_m} [u^i_m, u^j_n] \frac{\ptl g}{\ptl u^j_n} \, .
\end{equation}
(This is the coefficient of $ \la^0 $ in \eqref{e2.02}.)
Then $ \V $ becomes an (ordinary) Poisson algebra with an automorphism $ S, $ and such that 
\begin{equation}\label{e4.12}
[u^i_m, u^j_n] = 0 \text{ for } |m-n| \gg 0.
\end{equation}
\item[(b)] Conversely, if $ \V $ has a structure of a Poisson algebra with $ S $-invariant bracket $ [. \, , \, .], $ and satisfying the locality property \eqref{e4.12}, then, letting
\[ \{ u^i_\la u^j \} = \sum_{k\in \ZZ} \la^k [u^i_k, u^j_0] \]
and extending by the master formula \eqref{e2.02}, endows $ \V $ with the structure of a multiplicative PVA. The corresponding Hamiltonian operator is
\[ H(S) = \left( \sum_{k\in \ZZ} [u^j_k, u^i_0] S^k  \right)_{i,j\in I} \, . \]
\end{enumerate}
\end{proposition}
\begin{proof}
Straightforward verification.
\end{proof}
\begin{remark}\label{rem4.5}
In the case of one difference variable $ u $ the multiplicative $ \la $-bracket is defined by $ \{ u_\la u  \} = \sum_{k \in \ZZ} f_k \la^k, $ and, by \eqref{e4.10} the corresponding Poisson bracket becomes
\[ [u_m, u_n] = S^n f_{m-n}, \ m,n \in \ZZ.   \]
For example, in case of the general type multiplicative $ \la $-bracket
of order $N$, given by $f_k=c_kg(u)g(u_k), c_{-k}=-c_k, k=1,...,N$, the corresponding Poisson bracket is
\begin{equation}\label{e4.13}
[u_m, u_n] = c_{m-n} g(u_m) g(u_n). 
\end{equation}
In the case of the complementary multiplicative $ \la $-bracket of degree 2 we have
\[ f_1 = g(u)g(u_1) (F(u) + F(u_1)), \quad f_2 = g(u) g(u_2) F(u_1),  \]
and the corresponding Poisson bracket is
\begin{equation}\label{e4.14}
\begin{aligned}
\left[u_m, u_n\right] & = \pm g(u_m) g(u_n) F(u_{n \pm 1}) \text{ if } m-n = \pm 2, \\
& = \pm g(u_m) g(u_n) (F(u_m) + F(u_n)) \text{ if } m-n = \pm 1,\\
& = 0 \text{ otherwise.}
\end{aligned}
\end{equation}
Bracket \eqref{e4.13} for $N=1$ is compatible with bracket \eqref{e4.14}.
A linear combination of these brackets for $ g(u) = F(u) = u $ is the well-known Faddeev-Takhtajan-Volkov bracket \cite{FT86}.
\end{remark}

\section{The variational complex}
Let $ \V $ be an algebra of difference functions in the variables $ u^i, \ i \in I = \{ 1, \ldots, \ell \}. $ The \emph{basic de Rham complex} $ \tilde{\Omega} = \tilde{\Omega} (\V) $ is defined as a free commutative superalgebra over $ \V $ with odd generators $ \delta u^i_n, \ i \in I, n \in \ZZ. $ It has the same properties as the basic de Rham complex, studied in \cite{BDSK09}, Section 3. We recall here the most necessary of them.

The superalgebra $ \tilde{\Omega} $ consists of finite sums of the form
\begin{equation}\label{e5.01}
\tilde{\omega} = \sum_{\substack{i_1, \ldots, i_k \in I \\ m_1, \ldots, m_k \in \ZZ} } f^{m_1, \ldots, m_k}_{i_1, \ldots, i_k} \delta u^{i_1}_{m_1} \wedge \ldots \wedge \delta u^{i_k}_{m_k}, \quad f^{m_1, \ldots, m_k}_{i_1, \ldots, i_k} \in \V,  
\end{equation}
and has the usual (super)commutative product $ \wedge. $
This is a $ \ZZ_+ $-graded superalgebra: $  \tilde{\Omega} (\V) = \bigoplus_{k \in \ZZ_+} \tilde{\Omega}^k, $  where the grading is defined by letting $ \deg \V = 0, \ \deg \delta u^i_n = 1.$ It carries an odd derivation $ \delta $ of degree 1, defined by 
\[ \delta (\delta u^i_n) = 0, \quad \delta f = \sum_{\substack{i \in I \\ n \in \ZZ}} \frac{ \ptl f}{\ptl u^i_n} \delta u^i_n \text{ for } f \in \V. \]
It is immediate to check that $ \delta^2 = 0, $ hence we have cohomology of this complex
\[ H(\tilde{\Omega} (\V), \delta ) = \bigoplus_{k \geq 0} H^k (\tilde{\Omega} (\V), \delta) \, .\]

In the same way as in \cite{BDSK09} we show that the complex $ (\tilde{\Omega} (\V), \delta) $ is acyclic, provided that the algebra of difference functions $ \V $ is normal, as defined below. 
The algebra $ \V $ carries a filtration by subspaces 
\[ \V_{n,i} = \left\{ f \in \V \, \middle| \, \frac{\ptl f}{\ptl u^j_m} = 0
\,\text{ for } (m, j) > (n,i ) \, \text{in the lexicographical order} \right\}
. \]
The algebra $ \V $ is called \emph{normal}  if $ \frac{\ptl}{\ptl u^i_n} \V_{n,i} = \V_{n,i} $  for all $ i \in I, n \in \ZZ $.
\begin{theorem}\label{th5.1}
$ H^k (\tilde{\Omega} (\V), \delta ) = \delta_{k,0} \cF, $ provided that $ \V $ is normal. 
\end{theorem}

Next, we extend the automorphism $ S $ of the algebra $ \V $ to an automorphism of the superalgebra $ \tilde{\Omega} (\V), $ letting 
\[ S(\delta u^i_n) = \delta u^i_{n+1}, \ i \in I, n \in \ZZ, \]
and denote it again by $ S. $ It is immediate to check, using \eqref{e2.01}, that $ S  $ commutes with $ \delta,  $ hence $ (S-1) \tilde{\Omega} (\V) $ is a $ \delta $-invariant subspace, and we can define the \emph{reduced} complex
\begin{equation}\label{e5.02}
\Omega (\V) = \tilde{\Omega} (\V) / (S-1) \tilde{\Omega}(\V) = \bigoplus_{k \geq 0} \Omega^k (\V), 
\end{equation}
with the induced action of $ \delta. $ It is called the
variational complex.

In the same way as in \cite{BDSK09}, Section 3, using the long cohomology exact sequence, we prove 
\begin{theorem}\label{th5.2}
	\[ H^k (\Omega (\V), \delta) = \delta_{k,0} \cF / (S-1) \cF,  \]
provided that $ \V $ is normal. 
\end{theorem}

In the same way as in \cite{DSK09} and \cite{BDSK09} we have identifications of $  \Omega^0 (\V)  $ with $ \V / (S-1) \V, \ \Omega^1 (\V) $ with $ \V^\ell, $  $ \Omega^2 (\V) $ with the space of skewadjoint $ \ell \times \ell $ matrix difference operators,
and $ \Omega^k (\V) $ for $ k>2 $ with the space of skewsymmetric $ k $-difference operators. With these identifications we have explicit formulas for $ \delta $, similar to that in the above quoted papers. We shall need only the first two of them $ (f \in \V, F \in \V^\ell): $
\begin{equation}\label{e5.03}
\delta (\smallint f) = \frac{\delta \smallint f}{\delta u}, \quad \delta (F) = D_F (S) - D_F(S)^\ast \, , 
\end{equation}
where $ D_F(S) $ is
the difference Frechet derivative 
of $ F \in \V^\ell $, defined by
\eqref{e3.06a}.
As a result, we obtain the following corollary of Theorem \ref{th5.2}.
\begin{corollary}\label{cor5.3}
	Let $ \V $ be an algebra of difference functions. Then 
	\begin{enumerate}
		\item[(a)] $ \Ker \frac{\delta }{\delta u} \subseteq \cF + (S-1) \V, $ and we have the equality if $ \V $ is normal. 
		\item[(b)] If $ F \in \Im \frac{\delta}{\delta u} $, then $F$ is closed, i.e. $ D_F (S) = D_F(S)^\ast $. If $ \V $ is normal, the converse holds. \qed
	\end{enumerate}
\end{corollary}

\begin{examples}\label{ex5.4}
	\begin{enumerate}
		\item[(a)] The algebra $ \Pc_\ell $  is a normal algebra of difference functions with $ \cF =\cC = \FF.  $ Hence $ \Ker \frac{\delta}{\delta u } = \FF + (S-1) \Pc_\ell. $
		\item[(b)] The algebra $ \Pc_\ell [x] $ with $ S  $ extended from $ \Pc_\ell $ by $ S(x) =  x+1,  $ is normal with $ \cF = \FF [x] = (S-1) \cF.  $ Hence $ \Ker \frac{\delta }{\delta u} = (S-1) \Pc_\ell. $
		\item[(c)] The algebra $ \Pc_1 [ u^{-1}_n, \log u_n \, | \, n \in \ZZ] $ is a normal algebra of difference functions. 
	\end{enumerate}
\end{examples}

As explained in \cite{DSK13}, Lemma 4.3. any algebra of difference functions $ \V $ can be extended to a normal one, which can be taken to be a domain if $ \V $ is. 

The importance of the variational complex is revealed by the following theorem.
\begin{theorem}\label{th5.5}
Let $ H $ and $ K  \in \text{Mat}_{\ell \times \ell} \, \V [S. S^{-1}] $ be two compatible Hamiltonian difference operators and assume that $ K $  is non-degenerate (i.e. $ KM = 0 $ implies $ M = 0 $ for any $ M \in \text{Mat}_{\ell \times \ell} \, \V [S, S^{-1}] $). Let $ \xi_0, \xi_1, \xi_2 \in \Omega^1 (\V) = \V^\ell $ be such that 
\begin{equation}\label{e5.04}
K \xi_{n+1} = H \xi_n \text{ for } n = 0, 1. 
\end{equation}
If $ \xi_0 $ and $ \xi_1 $ are exact, them $ \xi_2 $ is closed (in the variational complex).
\end{theorem}

This theorem is well known in the theory of evolutionary PDE. Its simplest proof was
given in the framework of the theory of Dirac structures \cite{D93}, \cite{BDSK09}, \cite{DSK13}. A parallel theory of Dirac structures in the difference case can be developed without difficulty. In particular, this gives a proof of Theorem \ref{th5.5}.

The following symmetric bilinear form is used in the definition of a Dirac structure:
\begin{equation}\label{e5.05}
\V^\ell \times \V^\ell \rightarrow \bar{\V}, \ (F \, | \, G) = \smallint F \cdot G \, .\end{equation}
A proof, similar to that in \cite{BDSK09}, Proposition 1.3(a), shows that this form is non-degenerate. 
\section{The Lenard-Magri scheme.}
Let $ \V $ be an algebra of difference functions in $ \ell $ variables. Given two difference operators $ H(S) $ and $ K(S) : \V^{\ell}  \rightarrow \V^\ell,$ a sequence of elements $ \xi_0, \ldots, \xi_{N-1} \in \V^{ \ell }, \ N \geq 2, $ is called a \emph{Lenard-Magri sequence} if the following Lenard-Magri relations hold
\begin{equation}\label{e6.01}
K(S) \xi_{j} = H(S) \xi_{j-1}, \ j = 1, \ldots , N-1.
\end{equation}
For a difference operator $ L(S): \V^\ell \rightarrow \V^\ell $ define the bilinear form
\begin{equation}\label{e6.02}
\V^\ell \times \V^\ell \rightarrow \bar{\V},\, \langle F, G \rangle_L = (L(S) F \, | \, G).
\end{equation}
Note that this form is skewsymmetric if the operator $ L(S) $ is skewadjoint. 

The following theorem is analogous to that in the differential case, cf. \cite{D93}, \cite{BDSK09}.
\begin{theorem}\label{th6.1}
Let $ H(S) $ and $ K(S) $ be skewadjoint difference operators on $ \V^\ell, $ and let $ \xi_0, \ldots, \xi_{N-1} \in \V^\ell $ be a Lenard-Magri sequence. Then
\begin{enumerate}
\item[(a)] For all $ m,n \in \{ 0, \ldots, N-1\} $  one has 
\[ \langle \xi_m, \xi_n \rangle_H = 0 = \langle \xi_m, \xi_n \rangle_K \, . \]
\item[(b)] If $ \xi_j $ are exact, i.e. $ \xi_j = \frac{\delta h_j}{\delta u} $ for some $ h_j \in \V, \ j = 0, \ldots, N-1, $ then all the $ h_j $ are in involution with respect to both $ \la $-brackets, associated to the operators $ H $ and $ K $ via \eqref{e4.05}.
\item[(c)] If the following orthogonality condition holds:
\[ \text{span} \, \{ \xi_0, \ldots, \xi_{N-1} \}^\perp \subseteq \Im K(S) \]
where $ \perp $ stands for the orthogonal complement with respect to the bilinear form \eqref{e5.05}, then we can extend the sequence $ \xi_0, \ldots, \xi_{N-1} $ to an infinite Lenard-Magri sequence. 
\item[(d)] If $ \xi_0 $ and $ \xi_1 $ are exact, $ H(S) $ and $ K(S)  $ are compatible Hamiltonian operators and $ K(S) $ is non-degenerate (i.e. its kernel is finite-dimensional), then $ \xi_j = \frac{\delta \int h_j}{\delta u} $ for some
  $ h_j\in \V, \ j = 0, 1, \ldots, N-1, $ provided that $ \V $ is normal, and all
  these $ \int h_j $ are in involution with respect to the brackets, defined by \eqref{e4.04} for both $ H $ and $ K $. Furthermore, if $ N = \infty $ and the $\xi_j  $ span an infinite-dimensional subspace of $ \V^\ell, $ then
\[ \frac{du}{dt_j} = \{ \smallint h_j, u \}, \ j = 0, 1, \ldots,  \]
is a compatible integrable hierarchy of Hamiltonian difference equations.
	\end{enumerate}
\end{theorem}
\begin{proof}
a) goes back to Lenard and Magri, and can be found e.g. in \cite{BDSK09}, Lemma 2.6. (b) follows from (a) and \eqref{e4.04}. The proof of (c) is the same as in \cite{BDSK09}, Proposition 2.9. Claim (d) follows by Theorem \ref{th5.2} for $ k = 1 $ and Theorem \ref{th5.5}.
\end{proof}
\begin{remark}\label{rem6.2}
	Usually one begins a Lenard-Magri sequence with $ \xi_0 \in \Ker K(S).  $ If this sequence is infinite and one has another Lenard-Magri sequence that begins with $ \xi'_0 \in \Ker K(S),  $ then we also have $ \langle \xi_i, \xi'_j \rangle_{H \text{ or }K} = 0$ for all $ i, j, $ provided that both $ H(S) $ and $ K(S) $ are skewadjoint (see \cite{BDSK09}, Proposition 2.10(c)).
\end{remark}

\begin{remark}\label{rem6.3}
Let $ \xi_0 = 0 $ and $ \xi_1 = \frac{\delta h_1}{\delta u} \in \Ker K(S)$ for some $ h_1 \in \V. $ Then, under the assumptions on $ H $ and $ K $ of Theorem \ref{th5.5}, any $ \xi_2 \in \V^\ell,   $ such that $ K(S) \xi_2 = H(S) \xi_1. $ is closed. 
\end{remark}

The method of constructing infinite Lenard-Magri sequences, outlined in this section, works well in the PDE case, see \cite{BDSK09}, \cite{DSK13}. However, it doesn't work well for difference equations since it is difficult to verify the orthogonality condition of Theorem \ref{th6.1}(c). In the next section we use a different method.
\section{Integrable bi-Hamiltonian difference equations, associated to Hamiltonian operators of order $\leq 2$}

It follows from Theorem \ref{th2.1} that the following two difference operators are compatible (i.e. any their linear combination is a Hamiltonian difference operator):
\begin{equation}\label{e7.01}
K(S) = g(u) (g(u_1) S - g(u_{-1}) S^{-1} ),
\end{equation}
\begin{equation}\label{e7.02}
\begin{aligned}
H_2(S) = &  g(u) (g(u_1) (F(u) +F(u_1)S + F(u_1) g(u_2)  S^2)\\
& - g(u_{-1}) (F(u) + F(u_{-1}) S^{-1} + F(u_{-1}) g(u_{-2}) S^{-2}),\\
\end{aligned}
\end{equation}
where 
\begin{equation}\label{e7.03}
F' (u) g(u) = F(u), \ F'(u) \neq 0 \, .
\end{equation}
For $F=g=u$ these operators are listed in \cite{KMW13}.

Let 
\begin{equation}\label{e7.04}
\xi_0 =  \frac{1}{2g(u)}, \quad \xi_1 = F'(u) \, .
\end{equation}
It is straightforward to check the following:
\begin{equation}\label{e7.05}
K(S) \xi_0 = 0, \quad K(S) \xi_1 = H_2(S) \xi_0 \, . 
\end{equation}
It follows from \eqref{e7.05} that the difference equation
\begin{equation}\label{e7.06}
\frac{du}{dt_0} = g(u) (F(u_1) - F(u_{-1}))
\end{equation}
is bi-Hamiltonian (indeed, the RHS is equal to $ H_2(S) \xi_0 $). For $ F = u $ this is the Volterra chain \eqref{e1.14}; for F=$u^{\epsilon}$ this is the Kac-Moerbeke-Langmuir equation (see \cite{KMW13}).

We have the following compatible (by Theorem \ref{th6.1} and Proposition \ref{prop4.1}) with equation \eqref{e7.06} difference
equation $\frac{du}{dt_1} = H_2(S)\xi_1$, explicitly:  
\begin{equation}\label{e7.07}
  \frac{du}{dt_1}  =g(u)( F(u)F(u_1)+F(u_1)^2+F(u_1)F(u_2)-F(u)F(u_{-1})-F(u_{-1})^2-F(u_{-1})F(u_{-2})).
\end{equation}

We proceed to construct an infinite Lenard-Magri sequence, extending \eqref{e7.04} and \eqref{e7.05}. In fact, we shall consider a more general situation. Let $ \V $ be an algebra of difference functions in $ u, $ and let $ g(u) $ be an invertible element of $ \V $.

\begin{lemma}
\label{lem7.1}
Let $ K(S) $ be a difference operator, defined by \eqref{e7.01}, and let $ H(S) = \sum_k f_k S^k \in \V [S, S^{-1}] $ be an arbitrary skewadjoint difference operator.
Let $ \{ \xi_j  \}_{j = 0, \ldots, N-1} \subset \V, N\geq 2, $ be a Lenard-Magri sequence, such that $ \xi_0 = \frac{1}{2g(u)} \in \Ker K (S). $
Then
\begin{enumerate}
\item[(a)] $ \frac{1}{2 g(u)} H(S) \xi_{N-1} = (S-1)A_N, $ where 
\[ A_N = -g(u)g(u_{-1}) \sum_{\substack{i,j \geq 1 \\ i + j = N}} \xi_i S^{-1} (\xi_j) - \sum_{\substack{i,j \geq 0 \\ i+j = N-1}} \sum_{\substack{k \geq 1 \\ 0 \leq \ell \leq k-1}} S^\ell (f_{-k} \xi_i S^{-k} (\xi_j)). \]
\item[(b)] If $ 2S(A_N) = (1+S) g(u) \xi_N $ for some $ \xi_N \in \V, $ then $ K(S) \xi_N = H(S) \xi_{N-1}. $
\end{enumerate}
\end{lemma}
\begin{proof}
Let $ \xi = \sum_{j =0}^{N-1} \xi_j z^{-j}. $ Since $ K(S) \xi_0 = 0, $ we have $ K(S) \xi = \sum_{j=1}^{N-1} K(S) \xi_j z^{-j}. $ Using \eqref{e6.01}, we obtain $ K(S) \xi = \sum_{j=1}^{N-1} H(S) \xi_{j-1} z^{-j} = z^{-1} \sum_{j=0}^{N-2} H(S) \xi_j z^{-j}. $ Hence, multiplying by $ \xi, $ we obtain
\begin{equation}\label{e7.08}
\xi K(S) \xi = z^{-1} \xi H(S) \xi -\xi H(S) \xi_{N-1} z^{-N}.
\end{equation}
Note that, since $ K(S) = g(u) (S-S^{-1}) \circ g(u),  $ we have 
\[ \xi K(S) \xi = \xi g(u) S (g(u)\xi) -\xi g(u) S^{-1} (g(u) \xi) = S(\xi g(u) S^{-1} (\xi g(u))) - \xi g(u) S^{-1} (g(u) \xi). \]
Thus
\begin{equation}\label{e7.09}
\xi K(S) \xi = (S-1) (g(u) \xi S^{-1} (g(u)\xi)).
\end{equation}
Similarly, since $ H(S) = \sum_{k \geq 1} (f_k S^k - S^{-k} \circ f_k), $ we obtain
\begin{equation}\label{e7.10}
\xi H(S) \xi = \sum_{k \geq 1} (S^k-1)  ( \xi S^{-k} (f_k \xi)).
\end{equation}
Using \eqref{e7.09} and \eqref{e7.10}, we can rewrite \eqref{e7.08} as follows:
\begin{equation}\label{e7.11}
(S-1) (g(u)g(u_{-1}) \xi S^{-1} (\xi)) = -z^{-1} \sum_{\substack{k \geq 1 \\ 0 \leq \ell \leq k-1}} S^\ell (f_{-k} \xi S^{-k} (\xi)) - \xi H(S) \xi_{N-1} z^{-N}.
\end{equation}
Using that for any $ k \in \ZZ $ one has 
\[ \xi S^{-k} (\xi) = \sum_{p \geq 0} z^{-p} \sum_{\substack{0 \leq i,j \leq N-1 \\ i + j = p}} \xi_i S^{-k} (\xi_j), \]
we can rewrite \eqref{e7.11} as follows:
\[ \begin{aligned}
&(S-1) (g(u)g(u_{-1}) \sum_{p \geq 0} \sum_{\substack{0 \leq i,j \leq N-1 \\ i + j = p}} \xi_i S^{-1} (\xi_j) z^{-p}) \\
&= -(S-1) \sum_{p \geq 1} \sum_{\substack{0 \leq i,j \leq N-1 \\ i + j = p-1}} \sum_{\substack{k \geq 1 \\ 0 \leq \ell \leq k-1}} S^\ell (f_{-k} \xi_i S^{-k} (\xi_j))z^{-p} \\
&- \sum_{0 \leq j \leq N-1} \xi_j H(S) \xi_{N-1} z^{-(N+j)} \, .  
\end{aligned}
 \]
Looking at the coefficients of $ z^{-N} $ of this identity, we obtain claim (a) of the lemma. 

By claim (a) we have:
\[ \frac{1}{g(u)} H(S) \xi_{N-1} = 2 (S-1) S^{-1} S(A_N). \]
Hence, by the assumption of claim (b), we see that $ \frac{1}{g(u)} H(S) \xi_{N-1} = (1-S^{-1}) (1+S) g(u) \xi_N, $ and therefore $ H(S) \xi_{N-1} = g(u) (S-S^{-1}) \circ g(u) \xi_N = K(S) \xi_N,$ proving (b).
\end{proof}

Now we consider the special case of $ H(S)=H_2(S), $ given by \eqref{e7.02}, \eqref{e7.03}. Consider the difference operator
\[ D(S) = (S^2 \circ F(u_{-1}) +(S-1) \circ F(u) - S^{-1} \circ F(u_1)) \circ g(u).  \]
It is straightforward to check that the difference operator $ H_2(S), $ given by \eqref{e7.02}, \eqref{e7.03}, is expressed via $ D(S) $ as follows:
\begin{equation}\label{e7.12}
H_2(S) = g(u) \circ (1+S^{-1}) \circ D(S) \, . 
\end{equation}
\begin{lemma}
\label{lem7.2}
Let $ \eta_N = 2 S(A_N), $ where $ A_N $ is as in Lemma \ref{lem7.1}(a), let $ \zeta_N = D(S) \xi_{N-1}, $ and let $ \xi_N = \frac{1}{2 g(u)} (\eta_N - \zeta_N). $ Then
\begin{enumerate}
\item[(a)] $ S (\eta_N - \zeta_N) = \eta_N + \zeta_N \, .$
\item[(b)] $ (1+S) g(u) \xi_N = \eta_N \, . $
\end{enumerate}
\end{lemma}
\begin{proof}
By \eqref{e7.12} we have:
\[ \frac{1}{g(u)} H_2(S) \xi_{N-1} = (1 + S^{-1}) D(S) \xi_{N-1} = (1+S^{-1}) \zeta_N \, . \]
Hence, using Lemma \ref{lem7.1} (a), we see that 
\[ (1-S^{-1}) \eta_N = (1+S^{-1}) \zeta_N, \]
proving (a). The LHS of (b) is equal to 
\[ \half (1+S) (\eta_N - \zeta_N) = \half ((\eta_N - \zeta_N) + S(\eta_N-\zeta_N)) = \eta_N  \]
by (a), proving (b).
\end{proof}
\begin{proposition}
  \label{prop7.3} Let $ \V $ be an algebra of difference functions in $ u, $ let $ g(u) $ be an invertible function of $ \V, $ and $ F(u) \in \V $ be such that $ F'(u) \in \V $ and \eqref{e7.03} holds. Let $ K(S) $ and $ H_2(S) $ be the difference operators, defined by \eqref{e7.01} and \eqref{e7.02} respectively. Let $ N \geq 2 $ and let   $ \xi_0, \xi_1, \ldots, \xi_{N-1} $,
  where $\xi_0$ and $\xi_1$ are as in \eqref{e7.04},
  be a Lenard-Magri sequence for these operators, i.e. \eqref{e6.01} holds. Let $ \xi_N = \frac{1}{2 g(u)} (\eta_N - \zeta_N) $, as in Lemma \ref{lem7.2}. Then
\begin{enumerate}
\item[(a)] $ K(S) \xi_N = H_2(S) \xi_{N-1}, $ i.e. \eqref{e6.01} holds for $ \xi_0, \xi_1, \ldots, \xi_N. $
\item[(b)] For each $ j \geq 0 $ there exists $ h_j $ in an algebra of difference functions extension of $ \V,  $ such that $ \xi_j = \frac{\delta}{\delta u} \smallint h_j. $
\item[(c)] The hierarchy of difference equations 
\begin{equation}\label{e7.13}
\frac{du}{dt_{N-1}} = K(S) \xi_N = H_2(S) \xi_{N-1}, \ N = 1, 2, \ldots
\end{equation}
is an integrable system of compatible bi-Hamiltonian equations, for which $ \smallint h_j, \ j \geq 0,  $ are integrals of motion in involution. 
\end{enumerate}
\end{proposition}
\begin{proof}
Claim (a) follows from Lemma \ref{lem7.1} (b) and Lemma \ref{lem7.2} (b). Hence we have an infinite sequence $ \xi_0 , \xi_1, \xi_2, \ldots $ in $ \V$, where $\xi_0$ and $\xi_1$ are as in \eqref{e7.04}, satisfying the Lenard-Magri relation \eqref{e6.01}. The elements $ \xi_j $ for $ j \geq 1 $ are linearly independent, since, clearly, by \eqref{e6.01}, $ \ord \, \xi_{j+1} = \ord \, \xi_j +1  $ for $ j \geq 1 $ and $ \ord \, \xi_1 = 0. $

Since the difference operators $ K(S) $ and $ H_2(S) $ are compatible Hamiltonian operators, $ K(S) $ is non-degenerate and obviously both $\xi_0 , \xi_1$ are variational derivatives, by Theorem \ref{th6.1} (b) there exists $ h_j $ in a normal extension of $ \V, $ such that $ \xi_j = \frac{\delta}{\delta u} \smallint h_j $ (see Theorem \ref{th5.2} for $ k = 1 $), and all the $ \smallint h_j $ are in involution for both $ K(S) $ and $ H_2(S) $ by Theorem \ref{th6.1} (d). Hence, by the same theorem, \eqref{e7.13} is a hierarchy of compatible bi-Hamiltonian difference equations 
which are linearly independent. Hence this hierarchy is integrable. 
\end{proof}

The first few conserved densities for the hierarchy \eqref{e7.13} are as follows:
\[ 
\begin{aligned}
& h_0 = \smallint \frac{du}{2g(u)}, \ h_1 = F(u), \ h_2 = \half F(u)^2 + F(u) F(u_1), \\
& h_3 = \frac{1}{3} F(u)^3 + F(u)^2 F(u_1) + F(u)F(u_1)^2 + F(u) F(u_1) F(u_2) \, .
\end{aligned} \]
The corresponding variational derivatives $ \xi_j := \frac{\delta h_j}{\delta u} $ are as follows: 
\[ 
\begin{aligned}
 \xi_0 & =  \frac{1}{2g(u)}, \ \xi_1 = F'(u), \ \xi_2 = F'(u) (F(u_{-1}) + F(u) + F(u_1)), \\
 \xi _3 & = F'(u)( F(u_{-1}) (F(u_{-2}) + F(u_{-1}) +F(u))  +  ( F(u_{-1}) + F(u)) (F(u) + F(u_1)  ) \\
& + F(u_1) (F(u) + F(u_1) + F(u_2))) \, . \\
 \end{aligned} \]

\begin{conjecture}\label{con7.4}
The hierarchy of difference evolution equations \eqref{e7.13} coincides with the following hierarchy of Lax type equations 
\begin{equation}\label{e7.14}
\frac{dL}{dt_N} = \left[ \left(L^{2N+2}\right)_+, L \right], \ N = 0, 1, 2, \ldots
\end{equation}
where $ L = S+F(u)S^{-1} \in \V [S, S^{-1}] $ and the subscript +, as usual, means taking terms with non-negative powers of $ S. $ Moreover, the integrals of motion $ \int h_N $ constructed in Proposition \ref{prop7.3} are given by
\begin{equation}\label{e7.15}
\int h_N = \frac{1}{2N} \int \Res L^{2N}, \ N = 1, 2, \ldots, 
\end{equation}
where $ \Res : \V [S, S^{-1}] \rightarrow \V $ means taking the coefficient of $ S^0. $
\end{conjecture}

The fact that the Volterra equation \eqref{e1.14} coincides with \eqref{e7.14} for $ N= 0 $ and $ F(u) = u $, along with the integrals of motion \eqref{e7.15}, was pointed out in \cite{88}.


\begin{remark}\label{rem7.5}
By Remark \ref{rem2.7}, given $ a \in \cF, $ the Hamiltonian operator of order 1 with $ f_1 = \frac{g}{a} S \left( \frac{g}{a}\right) $ is compatible with the Hamiltonian operator of order 2 with 
\[ f_1 = gS(g) \left( \frac{S^{-1} (a)}{a} F + S \left( \frac{S(a)F}{a}  \right)  \right), \quad f_2 = gS^2 (g) S(F),  \]
where $ \frac{\ptl g}{\ptl u_i} = 0 = \frac{\ptl F}{\ptl u_i} $ for $ i \geq 1, \ g \frac{\ptl F}{\ptl u} = a F. $ We have a generalization of the Lenard-Magri scheme, studied in this section with 
\[ \xi_0 = \frac{a}{2 g(u)}, \quad \xi_1 = S(a) S^{-1} (a) \frac{\ptl F}{\ptl u}, \]
so that \eqref{e7.05} holds. As a result, we get a bi-Hamiltonian equation, generalizing \eqref{e7.06}:
\begin{equation}\label{e7.16}
\frac{du}{dt_0} = g(S^2 (a) S(F) - S^{-2} (a) S^{-1} (F)  ) \, .
\end{equation}
The same arguments show that the Lenard-Magri sequence extends to infinity, hence equation \eqref{e7.16} is integrable. Moreover, Conjecture \ref{con7.4}
extends to $ L = aS + S(a) FS^{-1}. $
(Note that \eqref{e7.16} can be rewritten as $ \frac{dL}{dt_0} = [(L^2)_{+}, L]$).
\end{remark}

\section{Classification of multiplicative PVA of order 3 in one variable.}
 
 \begin{theorem}
 \label{th8.1}
 Let $ \V $ be an algebra of difference functions in one variable $ u $, satisfying the assumptions of Theorem \ref{th2.1}. Then any multiplicative Poisson $ \la $-bracket on $ \V $ of order N=3 is either of general type, or is a linear combination of the multiplicative $\lambda$-bracket of complementary type
 $\{._\lambda .\}_{3,g,\epsilon}$
 (with non-zero coefficient) and the $\lambda$-bracket of general type
 $\{._\lambda .\}_{1,g}$. Explicitly, the latter is associated to functions $ g(u), F(u), G(u) \in \V $
 and constants $ a, c \in \cC, $ subject to the conditions
 \begin{equation}\label{e8.01}
 g(u) F'(u) = aF(u), \,g(u) G' (u) = a\epsilon G(u),\, \epsilon^2=-1,\, a\neq 0,
 \end{equation}
 and given by $ \{ u_\la u \} = \sum_{j=1}^{3} (\la^k - (\la S)^{-k}) f_k,$
 where 
 \begin{equation}\label{e8.02}
 \begin{aligned}
 f_1 & = g(u) g(u_1) (F(u) G(u_1) + c), \\
 f_2 & = g(u) g(u_2) (\epsilon F(u) G(u_1) +\epsilon^{-1} F(u_1) G(u_2)), \\
 f_3 & = g(u) g(u_3) F(u_1) G(u_2)\, . \\
 \end{aligned}
 \end{equation}
 \end{theorem}
\begin{proof}
  Recall that $ \{ u_\la u  \} $ is given by formula \eqref{e2.00}, where $ N = 4 $.
  The Jacobi identity M3 for $ a = b = c = u $ is formula \eqref{e2.03}, which is equivalent to the following nine identities (they are coefficients in \eqref{e2.03} of
  $ \la^3\mu^6$, $\la^3 \mu^5$, $\la^3 \mu^4$, $\la^2 \mu^5$, $\la^2 \mu^4$,
    $\la \mu^4$,   $\la \mu^3$,   $\la^2 \mu^3$, $\la \mu^2 $ respectively):
\begin{equation}\label{e8.03}
(S^3 f_3) \frac{\ptl f_3}{\ptl u_3} = f_3 S^3 \left( \frac{\ptl f_3}{\ptl u}\right),
\end{equation}
\begin{equation}\label{e8.04}
(S^3 f_2) \frac{\ptl f_3}{\ptl u_3} + (S^2f_3) \frac{\ptl f_3}{\ptl u_2} = f_3 S^3 \left( \frac{\ptl f_2}{\ptl u} \right),
\end{equation}
\begin{equation}\label{e8.05}
(S^3 f_1) \frac{\ptl f_3}{\ptl u_3} +(S^2 f_2) \frac{\ptl f_3}{\ptl u_2} + (Sf_3) \frac{\ptl f_3}{\ptl u_1} = f_3 S^3 \left( \frac{\ptl f_1}{\ptl u}  \right),
\end{equation}
\begin{equation}\label{e8.06}
(S^2 f_3) \frac{\ptl f_2}{\ptl u_2} = f_3 S^2 \left( \frac{\ptl f_3}{\ptl u_1} \right) + f_2 S^2 \left( \frac{\ptl f_3}{\ptl u} \right),
\end{equation}
\begin{equation}\label{e8.07}
(S^2 f_2) \frac{\ptl f_2}{\ptl u_2} + (Sf_3) \frac{\ptl f_2}{\ptl u_1} = f_3 S^2 \left( \frac{\ptl f_2}{\ptl u_1} \right) + f_2 S^2 \left(  \frac{\ptl f_2}{\ptl u} \right),
\end{equation}
\begin{equation}\label{e8.08}
(S f_3) \frac{\ptl f_1}{\ptl u_1} = f_3 S \left( \frac{\ptl f_3}{\ptl u_2} \right) + f_2 S \left( \frac{\ptl f_3}{\ptl u_1} \right) + f_1 S \left( \frac{\ptl f_3}{\ptl u} \right), 
\end{equation}
\begin{equation}\label{e8.09}
\begin{aligned}
& (S f_2) \frac{\ptl f_3}{\ptl u_3} + (Sf_1) \frac{\ptl f_3}{\ptl u_2} - f_1 \frac{\ptl f_3}{\ptl u} + (Sf_2) \frac{\ptl f_1}{\ptl u_1} + f_3 \frac{\ptl f_1}{\ptl u}\\ 
& = f_3 S \left( \frac{\ptl f_2}{\ptl u_2} \right) + f_2 S \left( \frac{\ptl f_2}{\ptl u_1}  \right) + f_1 S \left( \frac{\ptl f_2}{\ptl u}  \right),
\end{aligned}
\end{equation}
\begin{equation}\label{e8.10}
\begin{aligned}
& (S^2 f_1) \frac{\ptl f_3}{\ptl u_3} - (Sf_1) \frac{\ptl f_3}{\ptl u_1} - f_2 \frac{\ptl f_3}{\ptl u} + (S^2 f_1) \frac{\ptl f_2}{\ptl u_2} + (Sf_2) \frac{\ptl f_2}{\ptl u_1} + f_3 \frac{ \ptl f_2}{\ptl u} \\
& = f_3 S^2 \left( \frac{\ptl f_1}{\ptl u_1}  \right) + f_2 S^2 \left( \frac{\ptl f_1}{\ptl u} \right),
\end{aligned}
\end{equation}
\begin{equation}\label{e8.11}
(S f_1) \frac{\ptl f_1}{\ptl u_1} + (Sf_1) \frac{ \ptl f_2}{\ptl u_2} = f_1 \frac{\ptl f_2}{\ptl u} + f_1 S \left( \frac{\ptl f_1}{\ptl u}  \right) - f_2 \frac{ \ptl f_1}{\ptl u} + f_2 S \left( \frac{\ptl f_1}{\ptl u_1}  \right). 
\end{equation}
Equation \eqref{e8.03} can be rewritten as 
\[ \frac{1}{f_3} \frac{\ptl f_3}{\ptl u_3} = S^3 \left( \frac{1}{f_3} \frac{\ptl f_3}{\ptl u}  \right). \]
The LHS of this equation is a function of $ u, u_1, u_2, u_3, $ while the RHS is a function of $ u_3, u_4, u_5, u_6. $ Hence both sides are equal to $ \varphi (u_3),$
a function in $u_3$, so
\[ \frac{1}{f_3} \frac{\ptl f_3}{\ptl u_3} = \varphi (u_3), \quad \frac{1}{f_3} \frac{\ptl f_3}{\ptl u} = \varphi (u). \]
Hence $ \log f_3 = \int \varphi (u) du + \int \varphi (u_3) du_3 + (\text{function in } u_1, u_2), $ and, letting $ g(u) = \exp \int \varphi (u) du, $ we obtain for some $ A (u_1, u_2) \in \V: $
\begin{equation}\label{e8.13}
f_3 = g(u) g(u_3) A(u_1, u_2).
\end{equation}
Next, dividing both sides of equation \eqref{e8.04} by $ f_3, $ given by
\eqref{e8.13}, we obtain:
\begin{equation}\label{e8.14}
(S^3f_2) \frac{g' (u_3)}{g (u_3)} + \left( \frac{g(u_2)}{A (u_1, u_2)}  \frac{\ptl A (u_1, u_2)}{\ptl u_2}  \right) (A(u_3, u_4) g(u_5)) = S^3 \left( \frac{\ptl f_2}{\ptl u_2} \right)
\end{equation}
The first term and the second factor in the second term of the LHS, and the RHS are functions of $ u_3, u_4, u_5, $ while the first factor in the second term of the LHS is a function of $ u_1, u_2, $ hence it is a constant:
\begin{equation}\label{e8.15}
\frac{g(u_2)}{A (u_1, u_2)} \frac{\ptl A(u_1, u_2)}{\ptl u_2}= b \in \cC. 
\end{equation}
Hence $ \log A (u_1, u_2) = \int \frac{ b}{g(u_2)} du_2 + (\text{function in } u_1), $ and we conclude that
\begin{equation}\label{e8.16}
A(u_1, u_2) = G(u_2) F(u_1)
\end{equation}
for some functions $F(u)$ and $G(u)$ in $\V$.
Substituting \eqref{e8.16} in \eqref{e8.15}, we obtain 
\begin{equation}\label{e8.17}
g(u) G' (u) = b G(u).
\end{equation}
By \eqref{e8.13} and \eqref{e8.17} we obtain:
\begin{equation}\label{e8.18}
f_3 = g(u) g(u_3) F(u_1) G(u_2).
\end{equation}
Next. dividing equation \eqref{e8.06} by $ S^2 f_3 $ and substituting \eqref{e8.18}, we obtain
\begin{equation}\label{e8.19}
\frac{\ptl f_2}{\ptl u_2 } = g(u) F(u_1) G(u_2) \frac{g(u_3) F'(u_3)}{F(u_3)} + \frac{g' (u_2)}{g(u_2)} f_2, 
\end{equation} 
from which, as above, we conclude that
\begin{equation}\label{e8.20}
g(u) F'(u) = a F(u) \text{ for some } a \in \cC.
\end{equation}
Substituting \eqref{e8.20} in \eqref{e8.19}, we obtain
\begin{equation}\label{e8.21}
\frac{\ptl }{\ptl u_2} \left(  \frac{f_2}{g(u_2)} \right) = a g(u) F(u_1) \frac{G(u_2)}{g(u_2)} \, .
\end{equation}
Substituting \eqref{e8.15} and \eqref{e8.16} in \eqref{e8.14} and dividing both sides by $ g(u), $ we obtain
\begin{equation}\label{e8.22}
\frac{1}{g(u)} \frac{\ptl f_2}{\ptl u} - \frac{g' (u)}{g(u)^2} f_2 = b \frac{F(u)}{g(u)} G(u_1) g(u_2).
\end{equation}

First, consider the case $ a \neq 0. $ Then, using \eqref{e8.20}, equation \eqref{e8.22} can be rewritten as
\[ \frac{\ptl }{\ptl u} \left( \frac{f_2}{g(u)}\right) = \frac{b}{a} F' (u) G(u_1) g(u_2). \]
Hence we have for some $ B_1 (u_1, u_2) \in \V: $
\begin{equation}\label{e8.23}
f_2 = \frac{b}{a} g(u) g(u_2) F(u) G(u_1) + g(u) B_1 (u_1, u_2).
\end{equation}
Next, divide equation \eqref{e8.05} by $ f_3 $ and substitute \eqref{e8.17}, \eqref{e8.18}, and \eqref{e8.20} to obtain:
\begin{equation}\label{e8.24}
(S^3 f_1) \frac{g'(u_3)}{g(u_3)} + \left( \frac{b^2}{a} + a \right) F(u_2) G(u_3) g(u_4) + bB_1 (u_3, u_4) = S^3 \left( \frac{\ptl f_1}{\ptl u} \right).
\end{equation}
Since $ u_2 $ appears only in the second summand of \eqref{e8.24}, we conclude that 
\begin{equation}\label{e8.25}
\left( \frac{b}{a} \right)^2 + 1 = 0.
\end{equation}
Then \eqref{e8.24} becomes, after applying $ S^{-3}$:
\begin{equation}\label{e8.26}
\frac{g'(u)}{g(u)} f_1 + b B_1 (u, u_1) = \frac{ \ptl f_1}{\ptl u}.
\end{equation} 
Next, since $ b \neq 0 $ by \eqref{e8.25}, using \eqref{e8.17} and \eqref{e8.23}, we obtain
\[ \frac{\ptl }{\ptl u_2} \frac{B_1 (u_1, u_2)}{g(u_2)} = \frac{a}{b} F(u_1) G'(u_2),  \]
hence for some $ B(u)\in \V $ we obtain
\begin{equation}\label{e8.27}
B_1 (u_1, u_2) = g(u_2) \left( \frac{a}{b}  F(u_1) G(u_2) + B(u_1) \right),
\end{equation}
and, by \eqref{e8.23}, we have
\begin{equation}\label{e8.28}
f_2 = g(u) g(u_2) \left(  \frac{b}{a} F(u) G(u_1)  + \frac{a}{b} F(u_1) G(u_2) + B(u_1) \right).
\end{equation}
Next, dividing equation \eqref{e8.07} by $ G(u_2) G(u_3), $ we obtain, using \eqref{e8.18} and \eqref{e8.28},
\begin{equation}\label{e8.29}
\frac{a B(u_3) - B'(u_3)}{G(u_3)} = \frac{F(u_2)}{G(u_2)} \, \frac{b B(u_1)-B'(u_1)}{F(u_1)} \, .
\end{equation}
Since, by \eqref{e8.17}, \eqref{e8.20} and \eqref{e8.25}, $ \frac{F(u_2)}{G(u_2)} \notin \cC, $ we conclude from \eqref{e8.29} that $ aB(u) - B'(u)$ $=0= bB (u) - B'(u) , $ hence, by \eqref{e8.25}, we conclude that $ B(u) =0. $ Hence \eqref{e8.28} becomes
\begin{equation}\label{e8.30}
f_2 = g(u) g(u_2) \left(  \frac{b}{a} F(u) G(u_1) + \frac{a}{b} F(u_1) G(u_2)  \right).
\end{equation}
Using \eqref{e8.27} with $ B(u) = 0, $ equation \eqref{e8.26} becomes
\[ \frac{\ptl}{\ptl u} \left(  \frac{f_1}{g(u)}  \right) = a \frac{F(u)}{g(u)} G(u_1) g(u_1),  \]
hence, by \eqref{e8.20}, we obtain for some $ \phi (u) \in \V: $
\begin{equation}\label{e8.31}
f_1 = g(u) g(u_1) F(u)G(u_1) + g(u) \phi (u_1).
\end{equation}
Next, dividing both sides of \eqref{e8.08} by $ Sf_3 $, and using \eqref{e8.17}, \eqref{e8.18}, \eqref{e8.20}, \eqref{e8.25}, \eqref{e8.30} we obtain
\[ \frac{\ptl f_1}{\ptl u_1} = bg(u) F(u) G(u_1) + \frac{g'(u_1)}{g(u_1)} f_1, \]
hence, dividing by $ g(u_1), $ and using \eqref{e8.17}, we have 
\[ \frac{\ptl}{\ptl u_1} \frac{f_1}{g(u_1)} =  G'(u_1) g(u) F(u). \]
Hence for some $ \psi (u) \in \V  $ we have 
\[ \frac{f_1}{g(u_1)} = g(u) F(u) G(u_1) + \psi (u) \, . \]
Therefore, using \eqref{e8.31} we find that
\[ \frac{\psi (u)}{g(u)} = \frac{\phi (u_1)}{g(u_1)} \, . \]
Hence both sides are equal to a constant $ c, $ and $ \phi (u_1) = c g(u_1). $ So, from \eqref{e8.31} we obtain the first equation in \eqref{e8.02}. This completes the proof when $ a \neq 0. $

If $ a = 0, $ then, by \eqref{e8.20}, $ F(u) \in \cC, $ and we may assume, without loss of generality that $ F(u) = 1.  $ Then \eqref{e8.21} gives for some $ C_1 (u, u_1) \in \V $
\begin{equation}\label{e8.33}
f_2 = g(u_2) C_1 (u, u_1).
\end{equation} 
Using this, equation \eqref{e8.22} can be rewritten as \[ \frac{\ptl}{\ptl u} \frac{C_1 (u,u_1)}{g(u)} = b \frac{G(u_1)}{g(u)}, \]
hence for some $ C(u)\in \V $ we have
\[ C_1(u,u_1) = bg(u) \Phi(u) G(u_1) + g(u)C(u_1), \text{ where } \Phi (u) = \int \frac{du}{g(u)}. \]
Hence, \eqref{e8.33} becomes 
\begin{equation}\label{e8.34}
f_2 = g(u) g(u_2) (b \Phi (u) G(u_1) + C(u_1)   ). 
\end{equation}
Recall that, by \eqref{e8.18}, we have
\begin{equation}\label{e8.35}
f_3 = g(u) g(u_3) G(u_2).
\end{equation}
Note that $ G(u) \neq 0,  $ since $ N = 3.  $ Using \eqref{e8.34} and \eqref{e8.35}, equation \eqref{e8.07} becomes, after dividing by $ G(u_3) $
\[g(u_1) C'(u_1) - b C(u_1) = \frac{G(u_2)}{G(u_3)} (b^2 \Phi (u_2) G(u_3) + g(u_3) C' (u_3) ) \, .\]
Since the LHS is a function in $ u_1 $ and the RHS is a function of $ u_2, u_3, $ we conclude that both sides are equal to a constant $ c_1, $ in particular
  \[\frac{G(u_2)}{G(u_3)} ( b^2 \Phi (u_2) G(u_3) + g(u_3) C' (u_3)    ) = c_1.\]
This equation can be rewritten as 
\[b^2 \Phi (u_2) - \frac{c_1}{G(u_2)} = - \frac{g(u_3) C'(u_3)}{G(u_3)} \, .\]
Hence each side is a constant.
Hence $ b^2 \Phi (u_2) - \frac{c_1}{G(u_2)} $ is a constant, and applying it to $ \frac{\ptl }{\ptl u_2}, $ we obtain $ \frac{b^2}{g(u_2)} - c_1 \frac{G' (u_2)}{G(u_2)^2} = 0, $ which, by \eqref{e8.17} is equivalent to $ c_1b G(u_2) = b^2 G(u_2)^2. $ Hence $ G(u) \in \cC, $ and, by \eqref{e8.17}, $ b = 0. $ Without loss of generality we may assume that $ G(u) = 1. $ Thus, \eqref{e8.34} and \eqref{e8.35} become:
\begin{equation}\label{e8.39}
f_2 = g(u) g(u_2) C(u_1), \quad f_3 = g(u) g(u_3).
\end{equation}
Next, since $ \frac{\ptl f_3}{\ptl u_1} = \frac{\ptl f_3}{\ptl u_2} = 0, $ equations \eqref{e8.05} and \eqref{e8.08} become:
\[ (S^3 f_1) \frac{\ptl f_3}{\ptl u_3} = f_3 S^3 \left( \frac{\ptl f_1}{\ptl u} \right), \quad (Sf_3) \frac{\ptl f_1}{\ptl u_1} = f_1 S \left( \frac{\ptl f_3}{\ptl u} \right)\, . \]
Using \eqref{e8.39}, this becomes
\[ \frac{\ptl }{\ptl u} \left( \frac{f_1}{g(u)}  \right) = 0 = \frac{\ptl }{\ptl u_1} \left( \frac{f_1}{g(u_1)} \right) \, . \]
It follows that
\begin{equation}\label{e8.40}
f_1 = \beta g(u) g(u_1) \text{ for some } \beta \in \cC. 
\end{equation}
Next, equation \eqref{e8.10} gives, using \eqref{e8.39} and \eqref{e8.40}: $ (Sf_2) \frac{\ptl f_2}{\ptl u_1} = 0. $ Hence $ \frac{\ptl f_2}{\ptl u_1} = 0,  $ and, by \eqref{e8.39}, $ C(u) = \gamma \in \cC. $ Therefore $ f_2 = \gamma g(u) g(u_2), $ and due to \eqref{e8.39} and \eqref{e8.40}, we see that the multiplicative Poisson $ \la $-bracket in question is of general type. 
\end{proof}


\begin{remark}\label{rem8.9}
The coefficients of $ \la^{N} \mu^{N+j}, 1 \leq j \leq N, $ in the Jacobi identity \eqref{e2.03} for the multiplicative $ \la $-bracket \eqref{e1.9} of order $ N $ produces the following identities:
\begin{equation}\label{e8.36}
\sum_{k=j}^{N} S^{N+j-k} (f_k) \frac{\ptl f_N}{\ptl u_{N+j-k}} = f_N S^N \left( \frac{\ptl f_j}{\ptl u} \right), \ j= 1, \ldots, N.
\end{equation}
Using these equations and applying arguments similar to that proof of Theorems \ref{th2.1} and \ref{th8.1}, we obtain for $ \V $ as in Theorems \ref{th2.1}, \ref{th8.1}:
\begin{equation}\label{e8.37}
f_N = g(u) g(u_N) \prod_{i = 1}^{N-1} F_i (u_i), 
\end{equation}
where $ g(u), F_i(u) \in \V $ satisfy the relations
\begin{equation}\label{e8.38}
g(u) F'_i (u) = a_i F_i (u) \text{ for some } a_i \in \cC, \ i = 1, \ldots, N-1.
\end{equation}
However, to compute the $ f_j $ for $ 1 \leq j \leq N-1 $ is more difficult. 
\end{remark}

\section{Classification of multiplicative PVA of order 4 in one variable}

\begin{theorem}\label{th10.1}
Let $ \V $ be an algebra of difference functions in one variable $ u, $ satisfying the assumptions of Theorem \ref{th2.1}. Then any multiplicative Poisson $ \la $-bracket on $ \V $ of order 4 is one of the types (i), (ii), (iv), (v) or (vi) (see the introduction).
\end{theorem}
\begin{proof}
  Recall that the multiplicative $ \la $-bracket on $ \V $ is determined by $ \{ u_\la u \}, $ given by \eqref{e2.00}, where $N=4$.
  The condition on the functions $ f_j $ that one has to check is \eqref{e2.03}, and \eqref{e2.03} is equivalent to 16 equations, which are the coefficients of $ \la^m \mu^n $ with $ 1 \leq m \leq 4, m+1 \leq n \leq m + 4. $ By Remark \ref{rem8.9}, the four of these equations, corresponding to $ \la^4 \mu^{4 + j}, \ j = 1, \ldots, 4 $, give the following expression for $ f_4: $
\begin{equation}\label{e10.01}
f_4 = g(u)g(u_4) \Phi (u_1, u_2, u_3),
\end{equation}
where 
\begin{equation}\label{e10.02}
\Phi (u_1, u_2, u_3) = F_1 (u_1) F_2 (u_2) F_3 (u_3),
\end{equation}
\begin{equation}\label{e10.03}
g(u) F'_i (u) = a_i F_i(u) \text{ for some } a_i \in \mathcal{C}, \ i = 1,2,3.
\end{equation}
This expression for $ f_4 $ implies equation \eqref{e8.36} for $N= j = 4,  $ however one still has to study the remaining 15 equations. For this purpose we let
\begin{equation}\label{e10.04}
f_j (u, \ldots, u_j) = g(u) g(u_j) h_j (u, \ldots, u_j).
\end{equation}
By \eqref{e10.01} and \eqref{e10.03}, we have
\begin{equation}\label{e10.05}
h_4 = \Phi (u_1, u_2, u_3),
\end{equation}
\begin{equation}\label{e10.06}
g(u_i) \frac{\ptl h_4}{\ptl u_i} = a_i \Phi (u_1, u_2, u_3), \ i = 1,2,3.
\end{equation}
Then equation \eqref{e8.36} for $ N = 4, j = 3 $, becomes, using \eqref{e10.06} for $ i = 3 $:
\begin{equation}\label{e10.4/7}\tag{9.4/7}
	 a_3 S^3 (h_4) =  S^4 \left( g(u) \frac{\ptl h_3}{\ptl u} \right). 
\end{equation}
It follows that, applying $ S^{-3} $ and using \eqref{e10.03}, \eqref{e10.05}. we have
\begin{equation}\label{e10.07}
a_1 \frac{\ptl h_3}{\ptl u} = a_3 F'_1 (u) F_2 (u_1) F_3 (u_2).
\end{equation}

First, consider Case I: $ a_1 \neq 0. $ Then using \eqref{e10.07} and \eqref{e10.02} we obtain
\[ \frac{\ptl h_3}{\ptl u} = \frac{a_3}{a_1} F'_1 (u) F_2 (u_1) F_3 (u_2). \]
Hence we have for some $ A_1 (u_1, u_2, u_3) \in \V: $
\begin{equation}\label{e10.08}
h_3 = \frac{a_3}{a_1} \Phi (u, u_1, u_2) + A_1 (u_1, u_2, u_3).
\end{equation}

Next, equation \eqref{e8.36} for $ N=4, j=2, $ after the substitution \eqref{e10.04} and using \eqref{e10.06}, becomes
\begin{equation}\label{e10.4/6}\tag{9.4/6}
 S^4 \left( g(u) \frac{\ptl h_2}{\ptl u}  \right) = a_3 S^3 (h_3) + a_2 S^2 (h_4). 
\end{equation}
Applying to this equation $ S^{-3} $ and substituting \eqref{e10.08}, we obtain:
\[ S \left( g(u) \frac{\ptl h_2}{\ptl u} \right) = \left( \frac{a_3^2}{a_1} + a_2 \right) \Phi (u,u_1, u_2) + a_3 A_1 (u_1, u_2, u_3).    \]
The LHS of this equation is a function of $ u_1, u_2, u_3, $ but $ \frac{\ptl \Phi}{\ptl u}(u, u_1, u_2) \neq 0, $ hence
\begin{equation}\label{e10.09}
a^2_3 + a_1 a_2 = 0,
\end{equation}
and 
\begin{equation}\label{e10.10}
g(u) \frac{\ptl h_2}{\ptl u} = a_3 A_1 (u, u_1, u_2). 
\end{equation}

Next, equation \eqref{e8.36} for $ N = 4, j =1, $ using \eqref{e10.04} and \eqref{e10.06}, becomes
\begin{equation}\label{e10.4/5}\tag{9.4/5}
 S^4 \left( g(u) \frac{\ptl h_1}{\ptl u}   \right) = a_3 S^3 (h_2) + a_2 S^2 (h_3) + a_1 S(h_4). 
\end{equation}
Substituting $ h_3 $ and $ h_4, $ given by \eqref{e10.08} and \eqref{e10.05}, we obtain:
\[ S^2 \left( g(u) \frac{\ptl h_1}{\ptl u}  \right) = a_3 S (h_2) + \frac{a_2a_3 +a^2_1}{a_1} \Phi (u, u_1, u_2) + A_1 (u_1, u_2, u_3). \]
Since $ \Phi (u. u_1, u_2) $ depends on $ u $ and all other summands do not, we obtain 
\begin{equation}\label{e10.11}
a_2 a_3 + a^2_1 = 0,
\end{equation}
and 
\begin{equation}\label{e10.12}
S \left( g(u) \frac{\ptl h_1}{\ptl u}\right) = a_3 h_2 + a_2 A_1 (u, u_1, u_2). 
\end{equation}
It follows from \eqref{e10.09} and \eqref{e10.11} that 
\begin{equation}\label{e10.13}
a_j = \epsilon^{j-1} a_1, \text{ where } \epsilon^3 = -1
\end{equation}
Without loss of generality we may assume that $ a_1 = 1. $

We will use the equations on $h_1, h_2, h_3, h_4$, obtained by the substitution
\eqref{e10.04} in the coefficients of $\la^m \mu^n$ in \eqref{e2.03} for $N=4$; as before,
such an equation will be denoted by (9.m/n). First we use the following two of these equations:
\begin{equation}\label{e10.14}\tag{9.3/7}
S^3 (h_4) g(u_3) \frac{\ptl h_3}{\ptl u_3} = h_4 S^3 \left( g(u_1) \frac{\ptl h_4}{\ptl u_1} \right),
\end{equation}
\begin{equation}\label{e10.15}\tag{9.3/6}
S^3 (h_3) g(u_3) \frac{\ptl h_3}{\ptl u_3} + S^2 (h_4) g(u_2) \frac{\ptl h_3}{\ptl u_2} = h_4 S^3 \left(  g(u_1) \frac{\ptl h_3}{\ptl u_1} \right) + h_3 S^3 \left( g(u) \frac{\ptl h_3}{\ptl u} \right).
\end{equation}
Inserting \eqref{e10.08} in \eqref{e10.14} and using \eqref{e10.03} and \eqref{e10.13}, we obtain for some $ A_2 (u_1, u_2) \in \V: $
\begin{equation}\label{e10.16}
A_1 (u_1, u_2, u_3) = \epsilon^{-2} \Phi (u_1, u_2, u_3) + A_2 (u_1, u_2).
\end{equation}
Substituting this in \eqref{e10.08}, we have
\begin{equation}\label{e10.17}
h_3 = \epsilon^2 \Phi (u, u_1, u_2) + \epsilon^{-2} \Phi (u_1, u_2, u_3) + A_2 (u_1, u_2).
\end{equation}
Equation \eqref{e10.15} along with \eqref{e10.13} then gives, after some rearrangement, the following equation
\begin{equation}\label{e10.18}
\frac{A_2 (u_4, u_5) - g(u_4) \frac{\ptl A_2 (u_4, u_5)}{\ptl u_4}}{F_2 (u_4) F_3 (u_5)} = \frac{F_1 (u_3)}{F_3 (u_3)} \frac{\epsilon^2 A_2 (u_1, u_2) - g(u_2) \frac{\ptl A_2 (u_1, u_2)}{\ptl u_2}}{F_1 (u_1) F_2 (u_2)} \, .
\end{equation}
The LHS of this equation is a function of $ u_4, u_5, $ while the RHS is a product of a function, depending non-trivially on $ u_3 $ if $ \epsilon \neq -1 $ (resp. constant if $ \epsilon = -1 $) and a function of $ u_1, u_2. $ Hence
\begin{equation}\label{e10.19}
\text{both sides of \eqref{e10.18} are } 0 \, (\text{resp. } \in \cC) \text{ if } \epsilon \neq -1 \,(\text{resp. } \epsilon = -1).
\end{equation}
Consider separately two subcases of Case I, Case Ia: $ \epsilon \neq -1 $ and Case Ib: $ \epsilon = -1. $ 

By \eqref{e10.19} we have in Case Ia:
\[ \frac{\ptl}{\ptl u_j} \log A_2 (u_1, u_2) = \frac{\epsilon^{2(j-1)} }{g(u_j)} \text{ for } j = 1,2. \]
By \eqref{e10.03} and \eqref{e10.13} it follows that
\[ A_2 (u_1, u_2) = cF_1 (u_1) F_3 (u_2), \text{ for some } c \in \cC. \]
Substituting this in \eqref{e10.16}, \eqref{e10.17}, we obtain:
\begin{equation}\label{e10.20}
A_1 (u_1, u_2, u_3) = \epsilon^{-2} \Phi (u_1, u_2, u_3) + cF_1 (u_1) F_3 (u_2),
\end{equation}
\begin{equation}\label{e10.21}
h_3 = \epsilon^2 \Phi (u, u_1, u_2) + \epsilon^{-2} \Phi (u_1, u_2, u_3) + cF_1 (u_1) F_3(u_2). 
\end{equation}
By \eqref{e10.20} and \eqref{e10.10}, and using \eqref{e10.03}, we obtain for some $ B_1 (u_1, u_2) \in \V: $
\begin{equation}\label{e10.22}
h_2 = \Phi (u, u_1, u_2) + c \epsilon^2 F_1 (u) F_3 (u_1) + B_1 (u_1, u_2). 
\end{equation}
Substituting \eqref{e10.22} and \eqref{e10.20} in \eqref{e10.12} and using \eqref{e10.13}, we obtain
\begin{equation}\label{e10.23}
g(u) \frac{\ptl h_1}{\ptl u} = a_3 B_1 (u, u_1).
\end{equation}

Next, we consider the equation
\begin{equation}\label{e10.2/6}\tag{9.2/6}
S^2 (h_4) g(u_2) \frac{\ptl h_2}{\ptl u_2} = h_4 S^2 \left( g(u_2) \frac{\ptl h_4}{\ptl u_2} \right) + h_3 S^2 \left( g(u_1) \frac{\ptl h_4}{\ptl u_1}  \right).
\end{equation}
Using \eqref{e10.17}, \eqref{e10.22} and \eqref{e10.03} this gives for some $ B(u) \in \V: $
\begin{equation}\label{e10.24}
B_1 (u_1, u_2) = c\epsilon^{-2} F_1 (u_1) F_3 (u_2) + B(u_1).
\end{equation}
Substituting this in \eqref{e10.22} and \eqref{e10.23} we get
\begin{equation}\label{e10.25}
h_2 = \Phi (u, u_1, u_2) + c (\epsilon^2 F_1 (u) F_3 (u_1) + \epsilon^{-2} F_1 (u_1) F_3 (u_2)   ) + B(u_1),
\end{equation} 
\begin{equation}\label{e10.26}
g(u) \frac{\ptl h_1}{\ptl u} = cF_1 (u) F_3 (u_1) + \epsilon^2 B(u).
\end{equation}
Since in Case Ia we have $ 1 + \epsilon^2 = \epsilon, $ by \eqref{e10.03} and \eqref{e10.13}, we can choose the $ F_i, $ such that $ F_2 (u) = F_1 (u) F_3 (u). $ Then the following identities hold:
\begin{equation}\label{e10.27}
\Phi (u, u_1, u_2) F_1 (u_2) F_3 (u_3) = \Phi (u_1, u_2, u_3) F_1 (u) F_3 (u_1),  
\end{equation}
\begin{equation}\label{e10.28}
F_1 (u) F_3 (u_1) F_1 (u_1) F_3 (u_2) = \Phi (u, u_1, u_2).
\end{equation}

Next, we consider the equation
\begin{equation}\label{e10.29}\tag{9.3/5}
\begin{aligned}
&S^3 (h_2) g(u_3) \frac{\ptl h_3}{\ptl u_3} + S^2 (h_3) g(u_2) \frac{\ptl h_3}{\ptl u_2}  + S(h_4) g(u_1) \frac{\ptl h_3}{\ptl u_1} \\
& = h_4 S^3 \left( g(u_1) \frac{\ptl h_2}{\ptl u_1} \right) + h_3 S^3 \left( g(u) \frac{\ptl h_2}{\ptl u} \right).  
\end{aligned}
\end{equation}
Substituting in this equation \eqref{e10.21} and \eqref{e10.25}, we obtain, by making use of \eqref{e10.27}: $ B(u_4) = g(u_4) B' (u_4).  $
It follows that $ B(u) = b F_1 (u),  $ for some $ b \in \cC. $ Substituting this in \eqref{e10.24}, \eqref{e10.25}  and \eqref{e10.26}, we obtain
\begin{equation}\label{e10.30}
B_1 (u_1, u_2) = c \epsilon^{-2} F_1 (u_1) F_3 (u_2) + b F_1 (u_1),
\end{equation} 
\begin{equation}\label{e10.31}
h_2 = \Phi (u, u_1, u_2) + c (\epsilon^2 F_1(u) F_3 (u_1) + \epsilon^{-2} F_1 (u_1) F_3 (u_2)  ) + b F_1 (u_1),
\end{equation}
\begin{equation}\label{e10.32}
g(u) \frac{\ptl h_1}{\ptl u} = c F_1 (u_1) F_3 (u_1) + b \epsilon^2 F_1 (u).
\end{equation}
Dividing \eqref{e10.32} by $ g(u)  $ and   using \eqref{e10.03} and \eqref{e10.13}, we obtain for some $ P(u) \in \V: $
\begin{equation}\label{e10.33}
h_1 = cF_1 (u) F_3 (u_1) + b \epsilon^2 F_1 (u) + P(u_1).
\end{equation}

Next, we consider the equation, coming from the coefficient of $ \la \mu^5, $ which, using that $ g(u_j) \frac{\ptl h_4}{\ptl u_j} = a_j h_4, \ j = 1,2,3 $ (see \eqref{e10.02}--\eqref{e10.04}), and dividing by $ S(h_4),  $ becomes:
 \begin{equation}\label{e10.1/5}\tag{9.1/5}
 g(u_1) \frac{\ptl h_1}{\ptl u_1} = \epsilon^2 h_4 + \epsilon h_3 + h_2. 
 \end{equation}
Using \eqref{e10.21} and \eqref{e10.31}, this can be rewritten as
\[ g(u_1) \frac{\ptl h_1}{\ptl u_1} = c \epsilon^2 F_1 (u) F_3 (u_1) + bF_1 (u_1). \]
Using \eqref{e10.33}, this becomes: $ g(u_1) P'(u_1) = bF_1 (u_1),  $ which, using \eqref{e10.03}, becomes $ P' (u_1) = bF'_1 (u_1). $ Hence $ P(u_1) = bF_1 (u_1) + \al $ for some $ \al \in \cC. $
Substituting this in \eqref{e10.33}, we obtain 
\begin{equation}\label{e10.34}
h_1 = cF_1 (u) F_3 (u_1) + b \epsilon^2 F_1 (u) + b F_1 (u_1) + \al . 
\end{equation}

Next, we consider the equation, coming from the coefficient of $ \la^2 \mu^5, $ which is a differential equation on $ h_1, h_3 $ and $ h_2. $ Using the expressions \eqref{e10.05}, \eqref{e10.21} and \eqref{e10.31} of this functions and making use of \eqref{e10.27}, after some calculations, this equation becomes 
\[ b(1-\epsilon^2) \Phi (u_2, u_3, u_4) F_1 (u_1) = 0. \]
Since we are in Case Ia, it follows that $ b = 0. $ Hence, by \eqref{e10.34}, we have
\begin{equation}\label{e10.35}
h_1 = cF_1 (u) F_3 (u_1) + \al.
\end{equation}

In the same way, we can see that the coefficient of $ \la \mu^4 $ leads to the following equation:
\[ (c^2 + \al) (\epsilon - \epsilon^2) \Phi (u_1, u_2, u_3) = 0. \]
It follows that $ \al = -c^2. $ Hence the expressions \eqref{e10.05}, \eqref{e10.21}, \eqref{e10.31}, and \eqref{e10.35}, show that Case Ia produces type (vi) from the introduction of order 4 multiplicative PVA. Note that all the remaining five equations on the $ h_j $ are equivalent to the equations $ \al = -c^2. $

Now we turn to the Case Ib: $ a_1 \neq 0, \epsilon = -1. $ Again, we may assume that $ a_1=1, $ so that
\begin{equation}\label{e10.36}
a_j = (-1)^{j-1} \text{ for } j = 1,2,3.
\end{equation}
Hence we may assume that
\begin{equation}\label{e10.37}
F_1 (u) = F_3 (u) = F_2 (u)^{-1}.
\end{equation}

Recall that, by \eqref{e10.19}, both sides of \eqref{e10.18} are constant, which we denote by $ 2b_1. $ This gives the following equations:
\[ \frac{\ptl }{\ptl u_i} ( F_2 (u_1) F_2 (u_2) A_2 (u_1, u_2) ) = b_1 \frac{\ptl}{\ptl u_i} F_2 (u_i)^2, \ i = 1,2. \]
Hence we have for some $ b_2 \in \cC: $
\[ F_2 (u_1) F_2 (u_2) A_2 (u_1, u_2) = b_1 (F_2 (u_1)^2 + F_2 (u_2)^2) + b_2. \]
Using \eqref{e10.37}, we obtain from this
\begin{equation}\label{e10.38}
A_2 (u_1, u_2) = b_1 ( F_1 (u_2) F_2 (u_1) + F_1 (u_1) F_2 (u_2)   ) + b_2 F_1 (u_1) F_1 (u_2). 
\end{equation}
Substituting \eqref{e10.38} in \eqref{e10.16} and \eqref{e10.17} and using \eqref{e10.37}, we obtain 
\begin{equation}\label{e10.39}
A_1 (u_1, u_2, u_3) = \Phi (u_1, u_2, u_3) + b_1 (F_1(u_1) F_2 (u_2) + F_2 (u_1) F_1 (u_2) )+b_2F_1(u_1)F_1(u_2),
\end{equation}
\begin{equation}\label{e10.40}
\begin{aligned}
h_3  = & \Phi (u, u_1, u_2) + \Phi (u_1, u_2, u_3) + b_1 (F_1 (u_1) F_2 (u_2)\\
& + F_2 (u_1) F_1 (u_2)  ) + b_2 F_1 (u_1) F_1 (u_2). \\
\end{aligned}
\end{equation}
Substituting \eqref{e10.39} in \eqref{e10.10}, and using \eqref{e10.36}, we obtain, after integrating by $ u, $ for some $ B_1 (u_1, u_2): $
\begin{equation}\label{e10.41}
\begin{aligned}
h_2 = & F_1 (u) F_2 (u_1) F_3 (u_2) + b_1 (F_1(u) F_2 (u_1)- F_2 (u) F_1 (u_1)  )  + b_2 F_1 (u) F_1 (u_1)  \\
& + B_1 (u_1, u_2). \\
\end{aligned}
\end{equation}
Substituting \eqref{e10.39} and \eqref{e10.41} in \eqref{e10.12} we obtain
\begin{equation}\label{e10.42}
S \left( g(u) \frac{\ptl h_1}{\ptl u}  \right) = - 2 b_1 F_2 (u) F_1 (u_1) + B_1 (u_1, u_2).
\end{equation}
Since only the first term in the RHS depends on $ u, $ we deduce that $ b_1 = 0. $ Thus, formulas \eqref{e10.39}, \eqref{e10.40}, \eqref{e10.41}, and \eqref{e10.42} become:
\begin{equation}\label{e10.43}
A_1 (u_1, u_2, u_3) = \Phi (u_1, u_2, u_3) + b_2 F_1 (u_1) F_1 (u_2),
\end{equation}
\begin{equation}\label{e10.44}
h_3 = \Phi (u, u_1, u_2) + \Phi (u_1, u_2, u_3) + b_2 F_1 (u_1) F_1 (u_2)
\end{equation}
\begin{equation}\label{e10.45}
h_2 = \Phi (u, u_1, u_2) + b_2 F_1 (u) F_1 (u_1) + B_1 (u_1, u_2), 
\end{equation}
\begin{equation}\label{e10.46}
g(u) \frac{\ptl h_1}{\ptl u} =  B_1 (u, u_1).
\end{equation}

Now we return to equation     \eqref{e10.2/6}       .  Using that $ g(u_j) \frac{\ptl h_4}{\ptl u_j} = a_j h_4 $ for $ j = 1,2 $ (see \eqref{e10.05}, \eqref{e10.06}), and canceling this equation by $ S^2 (h_4),  $ we obtain, using \eqref{e10.36}:
\[ g(u_2) \frac{\ptl h_2}{\ptl u_2} = -h_4 + h_3. \]
Substituting in this equation \eqref{e10.05}, \eqref{e10.44} and \eqref{e10.45}, and using \eqref{e10.03} and \eqref{e10.37}, we obtain for some $ B(u) \in \V: $
\begin{equation}\label{e10.47}
B_1 (u_1, u_2) = b_2 F_1 (u_1) F_1 (u_2) + B(u_1).
\end{equation}
Using this, equations \eqref{e10.45} and \eqref{e10.46} can be written as follows:
\begin{equation}\label{e10.48}
h_2 = \Phi (u, u_1, u_2) + b_2 \left( F_1 (u) F_1 (u_1) + F_1 (u_1) F_1 (u_2) \right) + B(u_1),
\end{equation}
\begin{equation}\label{e10.49}
g(u) \frac{\ptl h_1}{\ptl u} = b_2 F_1 (u) F_1 (u_1) + B(u).
\end{equation}

Next, returning to equation    \eqref{e10.15}     , we
obtain, using 
\eqref{e10.05}, \eqref{e10.06}, \eqref{e10.44}, \eqref{e10.48}:
\[ 
\begin{aligned}
& -2b_2 \left( \Phi (u_2, u_3, u_4) F_1 (u_1) F_1 (u_2) - \Phi (u_1, u_2, u_3) F_1 (u_3) F_1 (u_4)   \right) \\
& = \Phi (u_1, u_2, u_3) \left( B(u_4) - g(u_4) B' (u_4) \right).
\end{aligned} \]
By \eqref{e10.02} and \eqref{e10.37} after multiplying by $ F_1 (u_2) F_1 (u_3)  $ this can be rewritten as
\[ 
\begin{aligned}
& -2b_2 F_1 (u_2)^3 F_1 (u_4) F_1 (u_1)\\
& = - 2b_2 F_1 (u_1) F_1 (u_3)^3 F_1 (u_4)
 + \left( B(u_4)-g(u_4) B' (u_4)  \right) F_1 (u_1) F_1 (u_3)^2. \\
\end{aligned} \]
Since the LHS of this equation depends on $u_2$ and the RHS does not, it follows that $ b_2 =0 $ and $ B(u) = g(u) B' (u). $ Hence, by \eqref{e10.03}, $ B(u) = cF_1 (u) $ for some $ c \in \cC. $ Thus, formulas \eqref{e10.44}, \eqref{e10.45}, and \eqref{e10.46} become respectively
\begin{equation}\label{e10.50}
h_3 = \Phi (u, u_1, u_2) + \Phi (u_1, u_2, u_3), 
\end{equation}
\begin{equation}\label{e10.51}
h_2 = \Phi (u, u_1, u_2) + c F_1 (u_1),
\end{equation}
\begin{equation}\label{e10.52}
g(u) \frac{\ptl h_1}{\ptl u} = c F_1 (u).
\end{equation}
Furthermore, by \eqref{e10.52} and \eqref{e10.30} we have for some $ P(u_1) \in \V: $
\begin{equation}\label{e10.53}
h_1 = cF_1 (u) + P(u_1). 
\end{equation}

Next, equation     \eqref{e10.1/5}        gives that $ g(u_1) \frac{\ptl h_1}{\ptl u_1} = h_4 - h_3 + h_2, $ hence, by \eqref{e10.50}--\eqref{e10.53}, we obtain that $ P(u_1) = cF_1 (u_1) + \al  $ for some $ \al \in \cC. $ Hence, by \eqref{e10.53},
\begin{equation}\label{e10.54}
h_1 = c \left( F_1 (u) + F_1 (u_1) \right) + \al. 
\end{equation}

In order to evaluate $ \al, $ we use the equation coming from the coefficient of $ \la \mu^4. $ It is straightforward to see, using \eqref{e10.50}, \eqref{e10.51} and \eqref{e10.54}, that this equation gives $ \al \Phi (u_1, u_2, u_3) = 0, $ hence $ \al = 0. $ Thus, by \eqref{e10.54}, \eqref{e10.51}, \eqref{e10.50}, and \eqref{e10.04}, the Case Ib produces type (iv) from the introduction. 

Now we turn to Case II: $ a_1 = 0. $ It follows from \eqref{e10.03} that $ F'_1 (u) = 0. $ Hence without loss of generality we may assume that $ F_1 = 1,  $ so that
\begin{equation}\label{e10.55}
h_4 = F_2 (u_2) F_3 (u_3).
\end{equation}
Then $ \frac{\ptl h_4}{\ptl u_1} = 0 $ and it follows from equation      \eqref{e10.14}       that 
\begin{equation}\label{e10.56}
\frac{\ptl h_3}{\ptl u_3} = 0.
\end{equation}

Recall that, by \eqref{e10.02}--\eqref{e10.04}, we have
\begin{equation}\label{e10.57}
g(u_j) \frac{\ptl h_4}{\ptl u_j} = a_j h_4, \ j = 1,2,3.
\end{equation}
Hence equation \eqref{e10.4/5} can be written, after applying $ S^{-2}, $ as
\begin{equation}\label{e10.58}
a_3 S(h_2) + a_2 h_3 = S^2 \left( g(u) \frac{\ptl h_1}{\ptl u}  \right)
\end{equation}
Since in the equation only the second term in the RHS may depend on $ u, $ we have
\begin{equation}\label{e10.59}
a_2 \frac{\ptl h_3}{\ptl u} = 0,
\end{equation}
\begin{equation}\label{e10.60}
a_2 \neq 0 \text{ implies } \frac{\ptl h_3}{\ptl u} = 0.
\end{equation}

Next, consider again equation     \eqref{e10.1/5}      .   Using \eqref{e10.57} it gives 
\begin{equation}\label{e10.61}
g(u_1) \frac{\ptl h_1}{\ptl u_1} = a_3 h_4 + a_2 h_3.
\end{equation}
Since $ h_j = h_j (u, u_1, \ldots, u_j) $, by \eqref{e10.55} and \eqref{e10.56} we conclude from this equation that $ a_3 h_4 = a_3 F_2 (u_2) F_3 (u_3) $ is independent on $ u_3. $ It follows from \eqref{e10.03} for $ i = 3  $ that 
\begin{equation}\label{e10.62}
a_3 = 0,
\end{equation}
and we may assume that $ F_3 = 1. $
Hence, by \eqref{e10.55},
\begin{equation}\label{e10.63}
h_4 = F_2 (u_2).
\end{equation}
By \eqref{e10.61} and \eqref{e10.62}, we have
\[ g(u_1) \frac{\ptl h_1}{\ptl u_1} = a_2 h_3. \]
Since the LHS of this equation depends only on $ u $ and $ u_1, $ we conclude that $ a_2 \frac{\ptl h_3}{\ptl u_2} =0. $ Hence
\begin{equation}\label{e10.64}
a_2 \neq 0 \text{ implies that } \frac{\ptl h_3}{\ptl u_2} = 0.
\end{equation}
By \eqref{e10.58} and \eqref{e10.62}, we obtain
\begin{equation}\label{e10.67}
  a_2 h_3 = S^2 \left( g(u) \frac{\ptl h_1}{\ptl u}\right).
  \end{equation}
Since the RHS is independent on $ u_1, $ we get 
\begin{equation}\label{e10.65}
a_2 \neq 0 \text{ implies that } \frac{\ptl h_3}{\ptl u_1} = 0.
\end{equation}

Next, consider equation   \eqref{e10.2/6}         . Using \eqref{e10.57} and \eqref{e10.63}, it can be written as 
 \begin{equation}\label{e10.66}
 g(u_2) \frac{\ptl h_2}{\ptl u_2} = a_2 F_2 (u_2).
 \end{equation}
 By equations \eqref{e10.4/6},  \eqref{e10.57} and \eqref{e10.63}, we have, after applying
 $ S^{-4}: $
 \begin{equation}\label{e10.67a}
 g(u) \frac{\ptl h_2}{\ptl u} = a_2 F_2 (u).
 \end{equation}

 We proceed to consider the following two remaining cases: Case IIa:
 $ a_1 = a_3 =0, a_2 \neq 0; $ Case IIb: $ a_1 = a_2 = a_3 = 0. $ In Case IIa we may assume that $ a_2 = 1.  $ By \eqref{e10.56}, \eqref{e10.59}, \eqref{e10.64}, and \eqref{e10.65}, in Case IIa we have
 \begin{equation}\label{e10.69}
 h_3 \in \cC.
 \end{equation}
 Then \eqref{e10.67} gives
 \begin{equation}\label{e10.70}
 g(u) \frac{\ptl h_1}{\ptl u} = h_3.
 \end{equation}

 Next, we consider the equation, coming from the coefficient of $ \la \mu^4. $ By \eqref{e10.69}, \eqref{e10.61}, \eqref{e10.63}, and \eqref{e10.70}, it gives, after applying $ S^{-1}: $
 \begin{equation}\label{e10.71}
 h_1 F_2 (u_1) + h^2_3 + h_3 F_2 (u_1) = 0. 
 \end{equation}
 Due to \eqref{e10.69}, applying to this equation $ \frac{\ptl }{\ptl u}, $ we obtain $ \frac{\ptl h_1}{\ptl u} F_2 (u_1) = 0. $ Multiplying this by $ g(u)  $ and using \eqref{e10.70}, we obtain
 \begin{equation}\label{e10.72}
 h_3 = 0.
 \end{equation}
 Hence, by \eqref{e10.71},
 \begin{equation}\label{e10.73}
 h_1 = 0.
 \end{equation}
 Considering again equation     \eqref{e10.29}           , we have by \eqref{e10.72}: $ h_4 S^3 \left( g(u_1) \frac{\ptl h_2}{\ptl u_1}  \right) =0, $ hence $ \frac{\ptl h_2}{\ptl u_1} = 0. $ Therefore, using \eqref{e10.66} and \eqref{e10.67a}, we obtain for some $ \al \in \cC, $
 \begin{equation}\label{e10.74}
 h_2 = F_2 (u) + F_2 (u_2) + \al. 
 \end{equation}
 Equations \eqref{e10.63}, \eqref{e10.72}, \eqref{e10.74}, and \eqref{e10.73} show that Case IIa produces Example (v) from the introduction.
 
 Finally, consider Case IIb. Due to \eqref{e10.03}, we may assume in this case that $ F_2 (u) = 1,$ hence, by \eqref{e10.63},
 \begin{equation}\label{e10.75}
 h_4 = 1.
 \end{equation}
 Then, by equation \eqref{e10.4/7}, we get $ \frac{\ptl h_3}{\ptl u} = 0. $ Recall that we also have \eqref{e10.56}; \eqref{e10.66} and \eqref{e10.67a}; \eqref{e10.61} and \eqref{e10.67}.
 These together give
 \begin{equation}\label{e10.76}
 \frac{\ptl h_i}{\ptl u} = \frac{\ptl h_i}{\ptl u_i}=0 \text{ for } i = 1,2,3.
 \end{equation}
 Next, equation \eqref{e10.15} gives, using \eqref{e10.75}:
 \begin{equation}\label{e10.77}
 g(u_2) \frac{\ptl h_3}{\ptl u_2} = S^2 \left( g(u_1) \frac{\ptl h_3}{\ptl u_1} \right).
 \end{equation}
 Due to \eqref{e10.76} for $ i = 3, $ we see that the LHS of this equation depends only on $ u_1 $ and $ u_2, $ while the RHS depends only on $ u_4 $ and $ u_5. $ Hence both sides are equal to a constant $ b: $
 \begin{equation}\label{e10.78}
 g(u_2) \frac{\ptl h_3}{\ptl u_2} = g(u_1) \frac{\ptl h_3}{\ptl u_1} =b.
 \end{equation}
 Next, equation   \eqref{e10.29}         , after using \eqref{e10.75} and \eqref{e10.78}, and applying $ S^{-2}, $ becomes
 \begin{equation}\label{e10.79}
 b(h_3 +1) = S \left( g(u_1) \frac{\ptl h_2}{\ptl u_1}  \right).
 \end{equation}
 By \eqref{e10.76} for $ i= 2 $, the RHS of \eqref{e10.79} depends only on $ u_2. $ Therefore, applying $ g(u_1) \frac{\ptl}{\ptl u_1} $ to both sides of \eqref{e10.79}, we obtain that $ b g(u_1) \frac{\ptl h_3}{\ptl u_1} = 0, $ hence, by \eqref{e10.78}, we have $b=0.$ Thus, by \eqref{e10.76}, \eqref{e10.78}, \eqref{e10.79}, we obtain that $ \frac{\ptl h_j}{\ptl u_i} = 0 $ for $ j = 1,2,3 $ and all $ i. $ Since, by \eqref{e10.75}, $ h_4 =1 $, we see that Case IIb produces the multiplicative $ \la $-bracket of general type. This completes the proof of Theorem \ref{th10.1}.
\end{proof}

\end{document}